\documentclass{amsart}
\usepackage{amsfonts,amssymb,amsmath,amsthm}
\usepackage{color}
\usepackage{graphicx}

\newcommand{\RE}{\mathbb R}
\newcommand{\f}{{\mathbf f}} 
\newcommand{\bu}{{\mathbf u}}
\newcommand{\bv}{{\mathbf v}}
\newcommand{\bV}{{\mathbf V}}
\newcommand{\bw}{{\mathbf w}}
\newcommand{\bW}{{\mathbf W}}

\newcommand{\bz}{{\mathbf z}}
\newcommand{\bJ}{{\mathbf J}}
\newcommand{\bphi}{\boldsymbol\phi}
\newcommand{\bpsi}{\boldsymbol\psi}

\newcommand{\balpha}{\boldsymbol\alpha}
\newcommand{\bn}{{\mathbf n}}
\newcommand{\bg}{{\mathbf g}}
\newcommand{\bL}{{\mathbf L}}
\newcommand{\bC}{{\mathbf C}}
\newcommand{\bH}{{\mathbf H}}

\newcommand{\diver}{\mathop{\rm{div}}\nolimits}

\newcommand{\bjump}[1]{\left[\!\!\left[#1\right]\!\!\right]}

\newcommand{\Ldo}{L^2_0(\Omega)}
\newcommand{\T}{\mathcal{T}}
\newcommand{\E}{\mathcal{E}}
\newcommand{\Gh}{\mathbf{G}_h}

\newcommand{\Ne}{N_{e}}
\newcommand{\hg}{h_\Gamma}
\newcommand{\tn}{|\!|\!|}
\newcommand{\normg}[1]{\tn\bg\tn_{#1,\Gamma}}

\newtheorem{theorem}{Theorem}[section]
\newtheorem{lemma}{Lemma}[section]
\newtheorem{proposition}{Proposition}[section]

\newtheorem{remark}{Remark}[section]
\newtheorem{problem}{Problem}[section]

\begin{document}
\title[Stokes equation with non-smooth data]
{Analysis of finite element approximations of Stokes equations
with non-smooth data}
\author{Ricardo G. Dur\'an}
\address{Departamento de Matem\'atica, Facultad de Ciencias Exactas y Naturales,
Universidad de Buenos Aires and IMAS, CONICET, 1428 Buenos Aires,
Argentina} \email{rduran@dm.uba.ar}
\urladdr{http://mate.dm.uba.ar/~rduran/}
\author{Lucia Gastaldi}
\address{DICATAM, Universit\`a di Brescia, Italy}
\email{lucia.gastaldi@unibs.it}
\urladdr{http://lucia-gastaldi.unibs.it}
\author{Ariel L. Lombardi}
\address{Departamento de Matem\'atica, Facultad de Ciencias Exactas, Ingenier\'\i a y Agrimensura,
Universidad Nacional de Rosario, and CONICET, Av. Pellegrini 250, 2000 Rosario, Argentina} \email{
ariel@fceia.unr.edu.ar} \urladdr{https://fceia.unr.edu.ar/~ariel/}
\thanks{The first and third authors are supported by ANPCyT under grant PICT2014-1771, by CONICET under grant PIP 11220130100184CO and by Universidad de
Buenos Aires under grant 20020160100144BA.
Second author gratefully acknowledges the hospitality of University of Buenos
Aires (Department of Mathematics) during her visit on March 2017 under the
project SAC.AD002.001.003/ ARGENTINA - CONICET - 050.000. She is partially funded
by IMATI-CNR and GNCS-INDAM.
Third author acknowledges IMATI (Pavia) and DICATAM, University of Brescia, for
the hospitality during his visit in November 2018 supported by the bilateral
project CONICET (Argentina) - CNR (Italy) and FLR 2015/2016 (University of
Brescia). He is also supported by Universidad Nacional de Rosario under grant ING568.
}
\subjclass{65N30, 65N15} \keywords{Stokes equations, finite elements, non
smooth data, a posteriori error analysis}

\begin{abstract}
In this paper we analyze the finite element approximation of the
Stokes equations with non-smooth Dirichlet boundary data. To define
the discrete solution, we first approximate the boundary datum by a smooth one
and then apply a standard finite element method to the regularized problem.

We prove almost optimal order error estimates for two regularization
procedures
in the case of general data in fractional order Sobolev spaces, and for the
Lagrange interpolation (with appropriate modifications at the
discontinuities) for piecewise
smooth data. Our results apply in particular to the classic lid-driven
cavity problem
improving the error estimates obtained in \cite{CW}.

Finally, we introduce and analyze an a posteriori error estimator. We prove
its reliability and efficiency, and show some numerical examples which
suggest that
optimal order of convergence is obtained by an adaptive procedure based on
our estimator.
\end{abstract}

\maketitle

\section{Introduction}
\label{se: setting}
\setcounter{equation}{0}

The goal of this paper is to analyze finite element approximations
of the Stokes equations with non smooth Dirichlet boundary data.
For the Laplace equation the analogous problem has been analyzed
in recent years in \cite{ANP,ANP2}.

Before explaining the problem and goals let us introduce some
notation. For $s$ a real number, $1\le p\le\infty$, and $D$ a
domain in $\RE^d$ or its boundary or some part of it, we denote by
$W^{s,p}(D)$ the Sobolev space on $D$, and by $\|\cdot\|_{s,p,D}$
and $|\cdot|_{s,p,D}$ its norm and seminorm respectively (see, for
example,~\cite{Adams,Adams2}). As it is usual, we write
$H^s(D)=W^{s,2}(D)$ and omit the $p$ in the norm and seminorm when
it is $2$. Moreover, bold characters denote vector valued
functions and the corresponding functional spaces. The notation
$(\cdot,\cdot)_D$ stands for the scalar product in $L^2(D)$ as well
as for the duality pairing between a Sobolev space and its dual; when
no confusion may arise the subscript indicating the domain is
dropped.

The subspace of $H^1(D)$ with zero trace on the boundary is
denoted as usual by $H^1_0(D)$, while $L^2_0(D)$ is the subspace
of $L^2(D)$ of functions with zero mean value.

Let $\Omega\subset\mathbf{R}^d$, $d=2,3$, be a Lipschitz domain
with boundary $\Gamma=\partial\Omega$ and denote by $\bn$ the
outward unit vector normal to the boundary.

We consider the Stokes problem
\begin{equation}
\label{eq: Stk_general}
\aligned
-\Delta \bu + \nabla p&= \f&&\mbox{in }\Omega\\
\diver \bu &=\eta&&\mbox{in }\Omega\\
\bu&=\bg&& \mbox{on }\Gamma.
\endaligned
\end{equation}
where $\f$, $\eta$ and $\bg$ are given data. If
$\f\in\bH^{-1}(\Omega)$, $\eta\in L^2(\Omega)$,
$\bg\in\bH^{1/2}(\Gamma)$, and the compatibility condition
\[
\int_\Gamma\bg\cdot\bn=\int_\Omega\eta
\]
is satisfied, existence and uniqueness of solution
$\bu\in\bH^1(\Omega)$ and $p\in L^2(\Omega)/\RE$ is a well known result (see
for example ~\cite[Page 31]{T}). Moreover, the following a priori
estimate holds true,
\begin{equation}
\label{eq: apriori}
\|\bu\|_{1,\Omega}+\|p\|_{L^2(\Omega)/\RE}\le
C\big(\|\f\|_{-1,\Omega}+\|\eta\|_{0,\Omega}+\|\bg\|_{1/2,\Gamma}\big).
\end{equation}

The classic analysis of finite element methods for this problem is
based on the variational formulation working with the spaces
$\bH^1(\Omega)$ for the velocity $\bu$ and $L^2(\Omega)$ for the
pressure $p$. If $\bg\notin\bH^{1/2}(\Gamma)$ then the solution
$\bu\notin\bH^1(\Omega)$, and therefore, that theory cannot be
applied. This situation arises in many practical situations. A
typical example is the so called lid-driven cavity problem where
$\Omega$ is a square and the boundary velocity $\bg$ is a
piecewise constant vector field which has jumps at two of the
vertices, and therefore, does not belong to $\bH^{1/2}(\Gamma)$.
However, this example is used in many papers as a model problem to
test finite element methods using some regularization of $\bg$
(although many times how the boundary condition is treated is not
clearly explained). Error estimates for this particular case were
obtained in \cite{CW,HTZ}. In \cite{CW}, the authors work with
$L^p$ based norms and use that $\bu\in\bW^{1,p}(\Omega)$ for
$1<p<2$. In \cite{HTZ} a particular regularization of the boundary
datum is considered.

More generally, we will consider boundary data
$\bg\in\bL^2(\Gamma)$ using some regularization of $\bg$ to define
the finite element approximation. In this way the a priori error
analysis is separated in two parts: the error due to the
regularization and that due to the discretization. We will analyze
the first error in general, assuming a given approximation of
$\bg$ and considering afterwards some particular regularizations
that can be used in practice.

For piecewise smooth boundary data, as in the case of the
lid-driven cavity problem, it is natural to use as an
approximation to $\bg$ its Lagrange interpolation at continuity
points with some appropriate definition at the discontinuities.
This is a particular regularization and so we can apply our
theory. We will show that this procedure produces an optimal order
approximation for the lid-driven cavity problem improving, in
particular, the result obtained in \cite{CW} where the order was
suboptimal. Let us remark that, since in this example the solution
belongs to $H^s(\Omega)$ for all $s<1$ (see \cite{AMZ,MZ}), the best expected
order for the error in the $L^2$ norm using quasi-uniform meshes is $O(h)$.

In the second part of the paper we introduce and analyze an a
posteriori error estimator of the residual type. We will prove
that the estimator is equivalent to appropriate norms of the
error. Numerical examples will show that an adaptive procedure
based on our estimator produce optimal order error estimates for
the lid-driven cavity problem.

Since \eqref{eq: Stk_general} with $\bg=0$ has been already
analyzed, we restrict ourselves to study the case $\f=0$ and
$\eta=0$, that is,
\begin{problem}
\label{problem 1 1} Given $\bg\in\bL^2(\Gamma)$ with
\begin{equation}
\label{eq: uno} \int_\Gamma \bg\cdot\bn=0,
\end{equation}
find $(\bu,p)$ such that
\begin{equation}
\label{eq: Stokes} \aligned
-\Delta \bu + \nabla p&= 0\qquad\text{in }\Omega\\
\diver \bu &=0\qquad\text{in }\Omega\\
\bu&=\bg\qquad \text{on }\Gamma.
\endaligned
\end{equation}
\end{problem}
The existence and uniqueness of solution is known. Indeed, we have

\begin{proposition}
\label{pr: existence} Let $\Omega$ be a Lipschitz convex polygon
or polyhedron, and $\bg\in\bL^2(\Gamma)$ satisfying the
compatibility condition \eqref{eq: uno}. Then the Stokes
system~\eqref{eq: Stokes} has a unique solution
$(\bu,p)\in\bL^2(\Omega)\times H^{-1}(\Omega)/{\mathbb R}$.

Moreover, there exists a constant $C$, depending only on $\Omega$,
such that
\begin{equation}
\label{eq: L2apriori} \|\bu\|_{0,\Omega}+\|p\|_{H^{-1}(\Omega)/\mathbb
R}\le C\|\bg\|_{0,\Gamma}.
\end{equation}
\end{proposition}

\begin{proof}

The existence of solution is proved in~\cite{HTZ} in the two dimensional case and
in \cite{FKV} in the three dimensional case. Actually, in \cite{HTZ}
the a priori estimate is proved only for smooth solutions but a standard density
argument, as the one we use below in Proposition \ref{pr: apriori}, can be applied
to obtain the general case.

On the other hand, in
\cite{FKV} it is not explicitly stated that $p\in H^{-1}(\Omega)$.
However, since $\bu\in\bL^2(\Omega)$ it follows immediately that
$\nabla p\in H^{-2}(\Omega)$ from which one can get $p\in
H^{-1}(\Omega)$ and \eqref{eq: L2apriori} (see \cite[page
317]{HTZ} and references therein). Let us also mention that
the method used in \cite{FKV} could also be applied in the two dimensional
case as it was done for the case of the Laplace equation in \cite{V}.
\end{proof}

The rest of the paper is organized as follows. In Section~\ref{se: apriori} we
introduce the finite element approximation which is based in replacing the
boundary datum $\bg$ by smooth approximations $\bg_h$. Then we develop the a
priori error analysis which is divided in two subsections.  In the first one we
estimate the error between the exact solution of the original problem and the
regularized one in terms of $\bg-\bg_h$. In the second subsection, considering
some appropriate computable approximations, we analyze the error due to the
finite element approximation of the regularized problem and prove a theorem
which gives a bound for the total error in terms of fractional order norms of
$\bg$.
Then, in Section \ref{se: pcwisesmooth} we consider the case of piecewise smooth data
approximated by a suitable modification of the Lagrange interpolation.
Section \ref{se: aposteriori} deals with a posteriori error estimates.
We introduce and analyze an error indicator for the regularized problem.
Finally, in Section \ref{se: examples}, we present some numerical examples
for the lid-driven cavity problem using two well known stable methods: the so called Mini element and the Hood-Taylor one.

\section{Finite element approximation and a priori estimates}
\label{se: apriori}
\setcounter{equation}{0}

In this section we introduce the finite element approximation to Problem
\ref{problem 1 1} and prove a priori error estimates.  As we have mentioned, in
general the solution $\bu$ of this problem is not in $\bH^1(\Omega)$ and so the
standard finite element formulation and analysis cannot be applied.
Therefore, to define the numerical approximation, we first approximate the
original problem by more regular ones and then solve these problems by standard
finite elements.  Consequently, our error analysis is divided in two parts that
we present in the following subsections. In the first one we analyze the error
due to the regularization, while in the second one the finite element
discretization error.

Given $\bg\in\bL^2(\Gamma)$, let $\bg_h\in\bH^\frac12(\Gamma)$ be approximations of $\bg$
such that
\begin{equation}
\label{eq: compatibility}
\int_\Gamma\bg_h\cdot\bn=0
\end{equation}
and
\begin{equation}
\label{gh tiende a g}
\lim_{h\to 0}\|\bg-\bg_h\|_{0,\Gamma}=0.
\end{equation}
Here $h>0$ is an abstract parameter which afterwords will be related to the
finite element meshes. The existence of approximations satisfying the
compatibility condition \eqref{eq: compatibility} is not difficult to prove.
Anyway we will construct explicit approximations later on using suitable
interpolations or projections.

For each $h$, we consider the following regularized problem:
find $\bu(h)$ and $p(h)$, such that

\begin{equation}
\label{eq: regu}
\aligned
-\Delta \bu(h) + \nabla p(h)&= 0&&\mbox{in }\Omega\\
\diver \bu(h) &=0&&\mbox{in }\Omega\\
\bu(h)&=\bg_h&&\mbox{on }\Gamma.
\endaligned
\end{equation}

This problem has a unique solution which, in view of \eqref{eq: apriori}, satisfies
\begin{equation}
\label{eq: apriori_reg}
\|\bu(h)\|_{1,\Omega}+\|p(h)\|_{L^2(\Omega)/\RE}\le
C\|\bg_h\|_{1/2,\Gamma}.
\end{equation}

The standard variational formulation of this regularized problem reads:
find $\bu(h)\in\bH^1(\Omega)$ with $\bu(h)=\bg_h$ on $\Gamma$ and
$p(h)\in L^2_0(\Omega)$ such that
\begin{equation}
\label{eq: weak}
\aligned
(\nabla\bu(h),\nabla\bv)-(\diver\bv,p(h))&=0&\forall\bv\in\bH_0^1(\Omega)\\
(\diver\bu(h),q)&=0&\forall q\in L^2_0(\Omega).
\endaligned
\end{equation}

\subsection{Analysis of the error due to the approximation of the boundary datum}

We will make use of the following well known result.

\begin{proposition}
\label{pr: reg R3}
Let $\Omega$ be a convex Lipschitz polygonal or polyhedral domain and
$\f\in \bL^2(\Omega)$. Then the system
$$
\aligned
-\Delta\bphi+\nabla q&=\f&&\mbox{in }\Omega\\
\diver\bphi&=0&&\mbox{in }\Omega\\
\bphi&=0&&\mbox{on }\Gamma.
\endaligned
$$
has a unique solution
$(\bphi,q)\in\bH^2(\Omega)\cap\bH^1_0(\Omega)\times H^1(\Omega)/\RE$ which
satisfies the following a priori estimate
\begin{equation}
\label{eq: est_hom-Dir}
\|\bphi\|_{2,\Omega}+\|q\|_{H^1(\Omega)/\RE}\le C \|\f\|_{0,\Omega}.
\end{equation}
\end{proposition}
\begin{proof}
This is proved in~\cite[Theorem 2]{KO} for $d=2$
and in~\cite[Theorem 9.20 (b)]{D} for $d=3$.
\end{proof}

The result given in the next lemma is known but we outline the proof
in order to make explicit the dependence of the involved constant on $s$.
We will denote by $\Gamma_i$, $1\le i\le N_e$, the edges or faces of $\Gamma$.

\begin{lemma}
\label{rem: extension}
There exists a constant $C$ independent of $s$ such that, for $0\le s<\frac12$,
\begin{equation}
\label{eq: extension2}
\|f\|_{-s,\Gamma_i}\le \frac C{1-2s} \|f\|_{-s,\Gamma}, \qquad
\forall f\in L^2(\Gamma).
\end{equation}
\end{lemma}
\begin{proof}
Given $\phi\in H^s(\Gamma_i)$ let $\widetilde\phi$ be its extension by $0$ to
$\Gamma$. Tracing constants in the proof of \cite[Th. 11.4 in Chapt.~1]{LM},
we can show that for $0\le s<\frac12$,
\begin{equation}
\label{extension por cero}
\|\widetilde \phi\|_{s,\Gamma}\le \frac C{1-2s}\|\phi\|_{s,\Gamma_i}
\qquad\forall \phi\in H^s(\Gamma_i),
\end{equation}
and then, we have
    \[
        \aligned
        \|f\|_{-s,\Gamma_i} &= \sup_{0\ne\phi\in H^s(\Gamma_i)} \frac{\int_{\Gamma_i}f\phi\,ds}{\|\phi\|_{s,\Gamma_i}}\\ &=\sup_{0\ne\phi\in H^s(\Gamma_i)} \frac{\int_{\Gamma_i}f\widetilde \phi\,ds}{\|\widetilde\phi\|_{s,\Gamma}} \frac{\|\widetilde\phi\|_{s,\Gamma}}{\|\phi\|_{s,\Gamma_i}}
        \endaligned
    \]
which yields
    \[
        \aligned
\|f\|_{-s,\Gamma_i} &\le \frac C{1-2s} \sup_{0\ne\phi\in H^s(\Gamma_i)}
\frac{\int_{\Gamma_i}f\widetilde \phi\,ds}{\|\widetilde\phi\|_{s,\Gamma}}\\
&\le\frac C{1-2s} \sup_{0\ne\phi\in H^s(\Gamma)}
\frac{\int_{\Gamma_i}f\phi\,ds}{\|\phi\|_{s,\Gamma}}
        \endaligned
    \]
    that is \eqref{eq: extension2}.
    \end{proof}

In the following proposition we estimate the error between the solutions
$(\bu,p)$ of~\eqref{eq: Stokes} and $(\bu(h),p(h))$ of~\eqref{eq: regu}
in the $L^2(\Omega)$-norm for the velocity and in $H^{-1}(\Omega)/\mathbb R$-norm for the pressure.

\begin{proposition}
\label{pr: apriori}
Let $\Omega$ be a convex Lipschitz polygonal or polyhedral domain and
$(\bu,p)$ and $(\bu(h),p(h))$ be the solutions of~\eqref{eq: Stokes}
and~\eqref{eq: regu}, respectively. Then, there exists a constants $C$,
independent of $h$, such that for $0\le s<\frac12$,

\begin{equation}
\label{eq: err_regu}
\|\bu-\bu(h)\|_{0,\Omega}+\|p-p(h)\|_{H^{-1}(\Omega)/\mathbb R}\le
\frac C{1-2s}\|\bg-\bg_h\|_{-s,\Gamma}.
\end{equation}
\end{proposition}
\begin{proof}
First we will estimate the $L^2(\Omega)$-norm of $\bv:=\bu-\bu(h)$.
Since $\Omega$ is convex, we know from Proposition~\ref{pr: reg R3}, that
there exist $\bphi\in\bH^2(\Omega)\cap\bH^1_0(\Omega)$ and $q\in
H^1(\Omega)\cap L^2_0(\Omega)$ solutions of the following system,
\[
\aligned
-\Delta\bphi+\nabla q&=\bv &&\quad\text{in }\Omega\\
\diver\bphi&=0 &&\quad\text{in }\Omega\\
\bphi&=0 &&\quad\text{on }\Gamma.
\endaligned
\]
Take $h_1$ another value of the parameter.
Then, taking into account~\eqref{eq: weak}, we have
\begin{equation*}
\aligned
(\bu(h_1)&-\bu(h),\bv)_\Omega=
\left(\bu(h_1)-\bu(h),-\Delta\bphi+\nabla q\right)_\Omega\\
&=
\left(\nabla(\bu(h_1)-\bu(h)),\nabla\bphi\right)_\Omega
-\left(\bu(h_1)-\bu(h),\frac{\partial\bphi}{\partial\bn}\right)_\Gamma
\\
&\quad  -\left(\diver(\bu(h_1)-\bu(h)),q\right)_\Omega
+ \left((\bu(h_1)-\bu(h))\cdot\bn,q\right)_\Gamma
\\
&= -\left(\bg_{h_1}-\bg_{h},\frac{\partial\bphi}{\partial\bn}\right)_\Gamma +
\left((\bg_{h_1}-\bg_{h})\cdot\bn,q\right)_\Gamma.
\endaligned
\end{equation*}
Summarizing,
\begin{equation}\label{eq: 4}
    \left(\bu(h_1)-\bu(h),\bv\right)_\Omega =
    -\left(\bg_{h_1}-\bg_{h},\frac{\partial\bphi}{\partial\bn}\right)_\Gamma +
    \left((\bg_{h_1}-\bg_{h})\cdot\bn,q\right)_\Gamma.
\end{equation}
Since, from \eqref{eq: L2apriori} and \eqref{gh tiende a g} we know that, for
$h_1\to 0$,
\[
    \|\bu-\bu(h_1)\|_{0,\Omega}\le C\|\bg-\bg_{h_1}\|_{0,\Gamma}\to 0,
\]
taking $h_1\to 0$ in \eqref{eq: 4}, we obtain
\begin{equation}
\label{eq: eq_for_u-u(h)}
    \left(\bu-\bu(h),\bv\right)_\Omega =
     -\left(\bg-\bg_h,\frac{\partial\bphi}{\partial\bn}\right)_\Gamma +
    \left((\bg-\bg_h)\cdot\bn,q\right)_\Gamma.
\end{equation}
We estimate the right hand side in terms of $\|\bg-\bg_h\|_{H^{-s}(\Gamma)}$. 
For the second term we note that while $q\in H^\frac12(\Gamma)$, due to the
discontinuities of $\bn$ we can not assure that $q\bn\in H^\frac12(\Gamma)$.
Therefore, with $0\le s<\frac12$, we have
\[
\aligned
\left(\bg-\bg_h,q\bn\right)_\Gamma &=\sum_{i=1}^{N_e}
\left(\bg-\bg_h,q\bn\right)_{\Gamma_i}
\le \sum_{i=1}^{N_e}\|\bg-\bg_h\|_{-s,\Gamma_i}
\left\|q\bn\right\|_{s,\Gamma_i}\\
&\le C\left(\sum_{i=1}^{N_e}\|\bg-\bg_h\|^2_{-s,\Gamma_i}\right)^{\frac12}
\|q\|_{1,\Omega}
\le \frac C{1-2s} \|\bg-\bg_h\|_{-s,\Gamma} \|q\|_{1,\Omega}
\endaligned
\]
where, in the last inequality, we have used~\eqref{eq: extension2}. With a
similar argument, we obtain for the first term in the right-hand side of
\eqref{eq: eq_for_u-u(h)} the estimate
\[
    \left(\bg-\bg_h,\frac{\partial\bphi}{\partial\bn}\right)_\Gamma\le \frac C{1-2s} \|\bg-\bg_h\|_{-s,\Gamma} \|\bphi\|_{2,\Omega}.
\]
Hence, from \eqref{eq: eq_for_u-u(h)} and the a priori
estimate \eqref{eq: est_hom-Dir} we have
\[
\aligned
\left\|\bu-\bu(h)\right\|^2_{0,\Omega}&=(\bu-\bu(h),\bv)\\
&\le \frac C{1-2s}\|\bg-\bg_h\|_{-s,\Gamma}
\left(\|\bphi\|_{2,\Omega}+\|q\|_{1,\Omega}\right)\\
&\le \frac C{1-2s}\|\bg-\bg_h\|_{-s,\Gamma}\|\bu-\bu(h)\|_{0,\Omega}
\endaligned
\]
and so,
\begin{equation}
\label{eq: u-u(h)}
\|\bu-\bu(h)\|_{0,\Omega}\le\frac C{1-2s}\|\bg-\bg_h\|_{-s,\Gamma}.
\end{equation}
Now, for the error in the pressure we have
\[
    \aligned
    \|p-p(h)\|_{{H^{-1}(\Omega)/\mathbb R}} &\le C\|\nabla(p-p(h))\|_{-2,\Omega}
= C\|\Delta(\bu-\bu(h))\|_{-2,\Omega} \\
&\le C\|\bu-\bu(h)\|_{0,\Omega}
\le \frac C{1-2s}\|\bg-\bg_h\|_{-s,\Gamma}
    \endaligned
\]
which concludes the proof.
\end{proof}

\begin{remark}
The estimate of the previous proposition can be compared with
\cite[Lemma 2.12]{ANP} where the corresponding result for the
approximation of a Poisson equation with a non smooth Dirichlet
boundary datum is considered. A constant independent of $s$ is
obtained in \cite{ANP}, while our estimate contains a factor $C/(1-2s)$.
Indeed, we could bound the first term in the
right-hand side of \eqref{eq: eq_for_u-u(h)} exactly as in
\cite{ANP}. However, the slightly worse factor $C/(1-2s)$ arises
due to the presence of the second term which involves the pressure $q$.
\end{remark}

\subsection{Analysis of the finite element approximation error}
Let $\{\mathcal T_h\}$, $h>0$,  be a family of meshes of $\Omega$, which is assumed to be
shape-regular, with $h$ being the maximum diameter of the elements in $\T_h$.
Each mesh $\T_h$ induces a mesh $\T_{\Gamma,h}$ along the
boundary fitted with the edges/faces $\Gamma_i$, $i=1,\dots,N_e$.

We consider a family of pairs $\bV_h=\bW_h\cap\bH^1_0(\Omega)$ and
$Q_h\subset L^2_0(\Omega)$ of finite element spaces, with
$\bW_h\subset\bH^1(\Omega)$, which are uniformly stable for the Stokes problem,
that is, the following inf-sup condition is satisfied for some
$\beta>0$ independent of $h$ (see, e.g., \cite[Chap.8]{BBF})
\[
\sup_{\bv_h\in \bV_h}
\frac{\left(q_h,\diver\bv_h\right)_{0,\Omega}}{\|\bv_h\|_{1,\Omega}}\ge
\beta\|q_h\|_{0,\Omega}, \qquad\forall q_h\in Q_h, \quad\forall h>0.
\]
Moreover, we assume that
\begin{equation}
\label{eq: hypo}
[\mathcal P_1(\mathcal T_h)\cap H^1(\Omega)]^d\subseteq \bW_h,
\qquad \mathcal P_0(\mathcal T_h)\subseteq Q_h,
\end{equation}
where $\mathcal{P}_k(\mathcal{T}_h)$ stands for the vector space of piecewise
polynomials of degree not grater than $k$ on the mesh $\mathcal{T}_h$.
In the following we shall use interpolant operators onto the discrete spaces
$\bW_h$ and $Q_h$.
For functions $\bphi\in\bH^2(\Omega)$, we define
$\bphi^I\in\bW_h$ as the continuous piecewise linear Lagrange interpolation of $\bphi$.
The following error estimates are well known:

\begin{equation}
\label{eq: interp}
\|\bphi-\bphi^I\|_{m,T}\le Ch^{2-m}|\bphi|_{2,T},\quad m=0,1
\quad\text{for all }\bphi\in\bH^2(\Omega).
\end{equation}
Let $P_0$ be the $L^2$-projection of $\Ldo$ onto $\mathcal P^0(\mathcal T_h)$,
it is well known that
\[
\|q-P_0q\|_{0,T}\le C h|q|_{1,\Omega} \quad\text{for all } q\in H^1(\Omega).
\]

From now on, we assume that
$\bg_h$ is the trace of a function $E\bg_h\in\bW_h$, for example, it is
enough to assume that $\bg_h$ is continuous and piecewise linear.
Moreover, it is known that $E\bg_h$ can be chosen such that
$\|E\bg_h\|_{1,\Omega}\le C\|\bg_h\|_{\frac12,\Gamma}$.

We consider the finite
element approximation of~\eqref{eq: weak} that reads:
find $\bu_{h}\in \bW_h$ and $p_h\in Q_h$ such that $\bu_h=\bg_h$
on $\Gamma$ and
\begin{equation}
\aligned
\left(\nabla\bu_h,\nabla\bv_h\right)-\left(\diver\bv_h,p_h\right)&=0
\qquad \forall\,\,\bv_h\in\bV_h\\
\left(\diver\bu_h,q_h\right)&=0\qquad\forall\,\, q_h\in Q_h.
\endaligned
\label{eq: pd}
\end{equation}
By taking $\bv_h=\bu_h-E\bg_h$ and $q_h=p_h$ in~\eqref{eq: pd}, and using the inf-sup
condition, we obtain existence and uniqueness and the estimate
\begin{equation}
\label{eq: est_disc}
\|\bu_h\|_{1,\Omega} + \|p_h\|_{0,\Omega} \le C\|\bg_h\|_{\frac12,\Gamma}.
\end{equation}
In the following proposition we estimate the finite element error in norms
corresponding with the ones used in Proposition~\ref{pr: apriori}.
\begin{proposition}
\label{pr: err_h}
Let $(\bu(h),p(h))\in\bH^1(\Omega)\times\Ldo$ with
$\bu(h)=\bg_h$ on $\Gamma$ and
$(\bu_h,p_h)\in \bW_h\times Q_h$ with $\bu_h=E\bg_h+\bu_{0h}$
be the solutions of~\eqref{eq: weak}
and~\eqref{eq: pd}, respectively. Then we have 
\begin{equation}
\label{eq: err_h}
\|\bu(h)-\bu_h\|_{0,\Omega} + \|p(h)-p_h\|_{H^{-1}(\Omega)/\mathbb R} \le
Ch\|\bg_h\|_{\frac12,\Gamma}.
\end{equation}
\end{proposition}
\begin{proof} Subtracting~\eqref{eq: pd} from~\eqref{eq: weak}, we get the
following error equations:
\begin{equation}
\label{cuatro}
\aligned
\left(\nabla(\bu(h)-\bu_h),\nabla\bv_h\right)
-\left(\diver\bv_h,p(h)-p_h\right)&=0&&\qquad\forall\bv_h\in V_h\\
\left(\diver(\bu(h)-\bu_h),q_h\right)&=0&&\qquad\forall q_h\in Q_h.
\endaligned
\end{equation}
In order to use a duality argument, we introduce the following system:
find $(\bphi,q)$ satisfying
\begin{equation}\label{aux3}
\aligned
-\Delta\bphi+\nabla q&=\bu(h)-\bu_h&&\qquad\mbox{in }\Omega\\
\diver\bphi&=0&&\qquad\mbox{in }\Omega\\
\bphi&=0&&\qquad\mbox{on }\Gamma.
\endaligned
\end{equation}
From Proposition~\ref{pr: reg R3},
$\bphi\in\bH^2(\Omega)\cap\bH^1_0(\Omega)$ and
$q\in H^1(\Omega)\cap L^2_0(\Omega)$ with the a priori
estimate~\eqref{eq: est_hom-Dir}. We have
\[
\|\bu(h)-\bu_h\|_{0,\Omega}^2 = \left(\bu(h)-\bu_h,-\Delta\bphi+\nabla q\right)
\]
then integration by parts, the error equations~\eqref{cuatro},
the approximation properties~\eqref{eq: interp}
and~\eqref{eq: est_disc}, the fact that
$\bu(h)=\bu_h=\bg_h$ on the boundary, and the a priori
estimates~\eqref{eq: apriori_reg} and~\eqref{eq: est_disc}, give
\[
\aligned
\|\bu(h)&-\bu_h\|_{0,\Omega}^2 \\
&=\left(\nabla(\bu(h)-\bu_h),\nabla\bphi\right)
-\left(\diver(\bu(h)-\bu_h),q\right)
-\left(\diver\bphi,p(h)-p_h\right)\\
&=\left(\nabla(\bu(h)-\bu_h),\nabla(\bphi-\bphi^I)\right)\\
&\qquad- \left(\diver(\bphi-\bphi^I),p(h)-p_h\right)
-\left(\diver(\bu(h)-\bu_h),q-P_0q\right)\\
&\le Ch\left(|\bphi|_{2,\Omega}+
|q|_{1,\Omega}\right)\|\nabla(\bu(h)-\bu_h)\|_{0,\Omega}
+Ch|\bphi|_{2,\Omega}\|p(h)-p_h\|_{0,\Omega}\\
&\le Ch\|\bu(h)-\bu_h\|_{0,\Omega}\left(\|\nabla\bu(h)\|_{0,\Omega}+
\|\nabla\bu_h\|_{0,\Omega}+\|p(h)\|_{0,\Omega}+\|p_h\|_{0,\Omega}\right)\\
&\le Ch\|\bu(h)-\bu_h\|_{0,\Omega}\|\bg_h\|_{\frac12,\Gamma}
\endaligned
\]
which provides the desired estimate for the velocity field
\begin{equation}
\label{eq: u(h)-u_h}
\|\bu(h)-\bu_h\|_{0,\Omega}\le Ch\|\bg_h\|_{\frac12,\Gamma}.
\end{equation}

Let us now estimate $p(h)-p_h$.  Since $p(h)-p_h\in L^2_0(\Omega)$, we have
\begin{equation}
\label{eq: tres}
\|p(h)-p_h\|_{H^{-1}(\Omega)/\mathbb R} =
\sup_{\substack{q\in H^1_0(\Omega)\\\int_\Omega q=0}}
\frac{\left(p(h)-p_h,q\right)}{\|q\|_{1,\Omega}}.
\end{equation}
Given $q\in H^1_0(\Omega)$ with $\int_\Omega q=0$, we know that there exists
$\bpsi\in\bH^2_0(\Omega)$ such that~\cite[Theorem 1]{S}
\begin{equation}
\label{aux2}
\diver\bpsi=q \quad\mbox{in }\Omega, \qquad
\|\bpsi\|_{2,\Omega}\le C\|q\|_{1,\Omega}.
\end{equation}
Then using the interpolant $\bpsi^I$ as in~\eqref{eq: interp}, and the error
equation~\eqref{cuatro}, we have
\[
\aligned
(p(h)-p_h,q) &=\left(p(h)-p_h,\diver\bpsi\right)\\
&=\left(p(h)-p_h,\diver(\bpsi-\bpsi^I)\right)
+\left(\nabla(\bu(h)-\bu_h),\nabla\psi^I\right)\\
&=\left(p(h)-p_h,\diver(\bpsi-\bpsi^I)\right)
-\left(\nabla(\bu(h)-\bu_h),\nabla(\bpsi-\bpsi^I)\right)\\
&\qquad+\left(\nabla(\bu(h)-\bu_h),\nabla\bpsi\right).
\endaligned
\]
Integrating by parts the last term, we have
\begin{equation}
\aligned
\left(p(h)-p_h,q\right) &=\left(p(h)-p_h,\diver(\bpsi-\bpsi^I)\right)
-\left(\nabla(\bu(h)-\bu_h),\nabla(\bpsi-\bpsi^I)\right) \\
&\qquad -\left(\bu(h)-\bu_h,\Delta\bpsi\right).
\label{eq: cinco}
\endaligned
\end{equation}
Then we obtain
\[
\aligned
\left(p(h)-p_h,q\right) &\le Ch|\bpsi|_{2,\Omega}\left(\|p(h)\|_{0,\Omega}
+ \|p_h\|_{0,\Omega}\right) \\
&\qquad +Ch\left(\|\nabla\bu(h)\|_{0,\Omega} +\|\nabla \bu_h\|_{0,\Omega}\right)
|\bpsi|_{2,\Omega}\\
&\qquad + \|\bu(h)-\bu_h\|_{0,\Omega}|\bpsi|_{2,\Omega}\\
&\le C\big[h\left(\|p(h)\|_{0,\Omega}  + \|p_h\|_{0,\Omega}
+\|\nabla\bu(h)\|_{0,\Omega} + \|\nabla \bu_h\|_{0,\Omega}\right) \\
&\qquad + \|\bu(h)-\bu_h\|_{0,\Omega} \big]\|q\|_{1,\Omega}.
\endaligned
\]
Substituting this inequality in~\eqref{eq: tres} implies
\[
\aligned
\|p(h)-p_h\|_{H^{-1}(\Omega)/\mathbb R}&\le Ch\left(\|p(h)\|_{0,\Omega}
+ \|p_h\|_{0,\Omega}+\|\nabla\bu(h)\|_{0,\Omega}
+\|\nabla \bu_h\|_{0,\Omega}\right)\\
&+\|\bu(h)-\bu_h\|_{0,\Omega}.
\endaligned
\]
Then the stability estimates \eqref{eq: apriori_reg} and \eqref{eq: est_disc}
joint with~\eqref{eq: u(h)-u_h}, give
\begin{equation}
\label{eq: p(h)-p_h}
\|p(h)-p_h\|_{H^{-1}(\Omega)/\mathbb R}\le Ch\|\bg_h\|_{\frac12,\Gamma}
\end{equation}
that together with \eqref{eq: u(h)-u_h} concludes the proof.
\end{proof}

The regularization of the boundary datum $\bg$ could be obtained by finite
element discretization.
By construction of the mesh $\T_h$, the boundary $\Gamma$ is subdivided into
boundary elements fitted with the edges/faces $\Gamma_i$, $i=1,\dots,\Ne$ and
$\T_{\Gamma,h}$ denotes the mesh along the boundary. Let $\hg$ be the maximum
diameter of the elements in $\T_{\Gamma,h}$ and define the discrete space on
the boundary as
\begin{equation}
\label{eq: Gh}
\Gh=\{\bz_h\in \mathbf{C}^0(\Gamma): \bz_h\in \mathcal{P}^1(E) \
\forall E\in\T_{\Gamma,h}\}.
\end{equation}
Then the function $\bg_h$ can be obtained either as the
$\bL^2(\Gamma)$-projection of $\bg$ onto the space $\Gh$, or using
the Carstensen interpolant $\bC_h\bg$ of $\bg$, see~\cite{C}, or by a suitable Lagrange interpolation, see Section \ref{se: pcwisesmooth}.
It is straightforward to check that both the $L^2$-projection and the
Carstensen interpolant provide approximations $\bg_h$ of $\bg$ which satisfy
the compatibility condition~\eqref{eq: compatibility}, while this is not always the case for the standard Lagrange interpolation.
Moreover we can show the following regularization error estimates for $\bg_h$
(see \cite[Lemmata 2.13 and A.2]{ANP}):

\begin{proposition}
\label{pr: Apel}
Let $\bg_h\in\Gh$ be either the piecewise linear Carstensen interpolant of
$\bg$ or the $L^2(\Gamma)$-projection on the continuous piecewise linear functions, then
we have
\begin{equation}
\label{eq: Apel_s}
\|\bg-\bg_h\|_{-s,\Gamma}\le Ch^{s+t}\|\bg\|_{t,\Gamma},\qquad\forall \bg\in H^t(\Gamma), s,t\in [0,1].
\end{equation}
We also have
\begin{equation}
\label{eq: Apel?}
\|\bg_h\|_{t,\Gamma}\le C\|\bg\|_{t,\Gamma}\qquad \forall \bg\in H^t(\Gamma),
t\in [0,1],
\end{equation}
where, for $t>0$, it is assumed that the mesh $\T_{\Gamma,h}$ is quasi-uniform.
\end{proposition}

\begin{proof}
    Inequality \eqref{eq: Apel_s} is proved in \cite[Lemma A.2]{ANP} for
$\bg_h$ being the Carstensen interpolant, and in \cite[Remark A.3]{ANP}
when $\bg_h$ is the $L^2$-projection on piecewise linear functions on $\Gamma$.

    Inequality \eqref{eq: Apel?} for $t=0$ is also proved in \cite{ANP}. For $t>0$ we can proceed as follows.

    Let $\Pi_h:\bH^1(\Gamma)\to\Gh$ be the Cl\'ement's operator such that for all $\bg\in\bH^1(\Gamma)$
    \[
        \|\bg-\Pi_h\bg\|_{r,\Gamma}\le
        c\left(\sum_K h_K^{2-2r}\|\bg\|_{1,K}^2\right)^{\frac12}
        \quad \text{for }r=0,1.
    \]
    Then, if the mesh is quasi-uniform we can use an inverse inequality and obtain
    \[
        \aligned
        \|\nabla\bg_h\|_{0,\Gamma}&\le \|\nabla\Pi_h\bg\|_{0,\Gamma}+
        \|\nabla(\bg_h-\Pi_h\bg)\|_{0,\Gamma}\\
        &\le c\|\bg\|_{1,\Gamma}+\frac{C_I}h\|\bg_h-\Pi_h\bg\|_{0,\Gamma}\\
        &\le
        c\|\bg\|_{1,\Gamma}+\frac{C_I}h\|\bg_h-\bg\|_{0,\Gamma}
        +\frac{C_I}h\|\bg-\Pi_h\bg\|_{0,\Gamma}\le C\|\bg\|_{1,\Gamma}.
        \endaligned
    \]
    Then by interpolation of Sobolev spaces (see, e.g., \cite[Prop. 14.1.5]{BS})
    we get~\eqref{eq: Apel?}.
\end{proof}

The bounds~\eqref{eq: err_regu} and~\eqref{eq: err_h} together with the
inequalities in Proposition~\ref{pr: Apel} give the following result.

\begin{theorem}
Let $\Omega$ be a convex polygonal or polyhedral domain. If the family of meshes
$\T_{\Gamma,h}$ is quasi-uniform and $\bg_h$ is given as in
Proposition~\ref{pr: Apel} then, for $0\le t<\frac12$ and $\bg\in H^t(\Gamma)$,
we have

\begin{equation}
\label{eq: estimate_t}
\|\bu-\bu_h\|_{0,\Omega}+\|p-p_h\|_{H^{-1}(\Omega)/\mathbb R}
\le C\,|\log h|h^{\frac12+t}\|\bg\|_{t,\Gamma}.
\end{equation}
\end{theorem}
\begin{proof}

From Proposition \ref{pr: apriori} and~\eqref{eq: Apel_s}, we have for
$0\le s<\frac12$
\[
\|\bu-\bu(h)\|_{0,\Omega}+\|p-p(h)\|_{H^{-1}(\Omega)/\mathbb R}\le
\frac C{1-2s}\,h^{s+t}\|\bg\|_{t,\Gamma}.
\]
Taking $s=1/2+1/\log h<1/2$ we obtain
\begin{equation}
\label{eq: Apel2}
\|\bu-\bu(h)\|_{0,\Omega}+\|p-p(h)\|_{H^{-1}(\Omega)/\mathbb R}\le
C\,h^{\frac12+t}\,|\log h| \|\bg\|_{t,\Gamma}.
\end{equation}
On the other hand, from Proposition \ref{pr: err_h},
\begin{equation}
\label{eq: Apel3}
\|\bu(h)-\bu_h\|_{0,\Omega}+\|p(h)-p_h\|_{H^{-1}(\Omega)/\mathbb R}\le
C\,h \|\bg_h\|_{\frac12,\Gamma}.
\end{equation}
Now, using an inverse inequality (\cite[Theorem 4.1]{DFGHS})
and~\eqref{eq: Apel?}, we obtain
\[
\|\bg_h\|_{\frac12,\Gamma}\le C h^{t-\frac12}\|\bg_h\|_{t,\Gamma}
\le C h^{t-\frac12}\|\bg\|_{t,\Gamma},
\]
which substituted in~\eqref{eq: Apel3} gives,
\begin{equation}
\label{eq: Apel5}
\|\bu(h)-\bu_h\|_{0,\Omega} + \|p(h)-p_h\|_{H^{-1}(\Omega)} \le
C\,h^{\frac12+t}\|\bg\|_{t,\Gamma}.
\end{equation}
Combining~\eqref{eq: Apel2} and~\eqref{eq: Apel5} we arrive at the
desired estimate~\eqref{eq: estimate_t}.
\end{proof}

\section{A priori error estimates for piecewise smooth boundary data}
\label{se: pcwisesmooth}
\setcounter{equation}{0}
In this section we analyze the approximation of piecewise smooth data, in
particular, our results can be applied to the lid-driven cavity problem. In
practice, the most usual way to deal with the non-homogeneous Dirichlet
boundary condition is to use the Lagrange interpolation or a simple
modification of it, to treat discontinuities and to obtain a compatible
approximation $\bg_h$. 

We shall use the following notation for the norm of $\bg$
\begin{equation}
\label{eq: normg}
\normg{k} =\left(\sum_{i=1}^{\Ne}\|\bg\|^2_{k,\Gamma_i}\right)^{\frac12}.
\end{equation}
In the following, we consider separately the case $d=2$ or $d=3$.

\subsection{Two dimensional case}
Let $\bg=(g_1,g_2):\Gamma\to\mathbb R^2$ be such that
$\bg|_{\Gamma_i}\in\bH^1(\Gamma_i)$ for $i=1,\ldots,\Ne$, where
$\Gamma_i$ are the boundary segments $\Gamma_i=[A_i,A_{i+1}]$
(with $A_{\Ne+1}=A_1$) and $A_i$, $i=1,\dots,\Ne$ are the boundary
vertices. 
We observe that $\bg\in\bH^s(\Gamma)$ with $0\le s<\frac12$. Indeed, let us set
$\bg_i=\bg|_{\Gamma_i}$. Since, for $0\le s<\frac12$,
$H^1(\Gamma_i)\subset H^s(\Gamma_i)$, we have that the extension by zero
$\tilde\bg_i$ of $\bg_i\in\bH^s(\Gamma_i)$ belongs to $\bH^s(\Gamma)$
(see~\cite[Th.11.4 in Chapt.1]{LM}) and, thanks to~\eqref{extension por cero}, 
\[
\|\tilde\bg_i\|_{s,\Gamma}\le \frac{C}{1-2s}\|\bg_i\|_{s,\Gamma_i}.
\]
Then $\bg=\sum_{i=1}^{\Ne} \tilde\bg_i$ belongs to $\bH^s(\Gamma)$, with
\begin{equation}
\label{eq: normsGamma}
\|\bg\|_{s,\Gamma}\le \frac{C}{1-2s}\normg{1}.
\end{equation}

We denote by $B_i$, $1\le i\le M$, the boundary nodes of
the mesh numbered consecutively and set $B_{M+1}=B_1$ (of
course these nodes depend on $h$ but we omit this in the notation
for simplicity) and $h_i=|B_{i+1}-B_i|$.
In principle, we would define $\bg_h$ as the continuous piecewise linear
vector field on $\Gamma$
such that $\bg_h(B_j)=\bg(B_j)$ if $\bg$ is continuous in $B_j$ and
$\bg_h(B_j)=\bg(B_j^-)$ or $\bg_h(B_j)=\bg(B_j^+)$, or some average of
these two values, if not. Notice that
$|\bg_h(B_j)|\le\|\bg\|_{L^\infty(\Gamma)}$.
However, in general, this definition does not
satisfy the compatibility condition \eqref{eq: compatibility}. We now
show how to enforce compatibility by a simple modification.
\begin{lemma}
\label{le: gh2D}
Given $\bg\in\bL^2(\Gamma)$ such that $\bg|_{\Gamma_i}\in\bH^1(\Gamma_i)$ for 
$i=1,\dots,N_e$,
there exists a piecewise linear function $\bg_h$ which is a modified Lagrange
interpolant of $\bg$ satisfying the compatibility
condition~\eqref{eq: compatibility}. Moreover, 
\begin{equation}
\label{eq: norminfgh2d}
\|\bg_h\|_{L^\infty(\Gamma)}\le C\normg{1}.
\end{equation}

\end{lemma}
\begin{proof}
We modify the definition of $\bg_h$ given above in some node~$B_k$.
For simplicity, let us choose this node different from all the vertices and
their neighbors, and such that $h_k$ is comparable to $h$. For each $j$, let
$\Gamma_{B_j}$ be the union of the two segments of $\T_{\Gamma,h}$ containing
$B_j$. Moreover, we set $\Gamma_V=\cup_{i=1}^{N_e}\Gamma_{A_i}$.

We want to define $\bg_h(B_k)$ in such a way that
\[
0=\int_\Gamma \bg_h\cdot\bn=\int_{\Gamma\setminus(\Gamma_V\cup\Gamma_{B_k})}\bg_h\cdot\bn
+ \int_{\Gamma_V}\bg_h\cdot\bn + \int_{\Gamma_{B_k}}\bg_h\cdot\bn,
\]
or, equivalently,
\[
\int_{\Gamma_{B_k}}\bg_h\cdot\bn
=-\int_{\Gamma\setminus(\Gamma_V\cup\Gamma_{B_k})}\bg_h\cdot\bn
-\int_{\Gamma_V}\bg_h\cdot\bn
= -\int_{\Gamma\setminus\Gamma_{B_k}}\bg_h\cdot\bn.
\]
But
\[
\aligned
     \int_{\Gamma_{B_k}}\bg_h\cdot\bn
&=\frac12 h_{k-1} \left[\bg(B_{k-1})+\bg_h(B_k)\right]\cdot\bn +
\frac12h_k\left[\bg_h(B_k)+\bg(B_{k+1})\right]\cdot\bn\\ &=
\frac12\left[h_{k-1}\bg(B_{k-1})+h_k\bg(B_{k+1})\right]\cdot\bn +
\frac12(h_{k-1}+h_k)\bg_h(B_k)\cdot\bn.
     \endaligned
\]
We introduce
\[
\aligned
L_1(\bg) &= -\int_{\Gamma\setminus\Gamma_{B_k}}\bg_h\cdot\bn
- \frac12\left[h_{k-1}\bg(B_{k-1})+h_k\bg(B_{k+1})\right]\cdot\bn.\\
\endaligned
\]
Notice that the integral
$\int_{\Gamma\setminus\Gamma_{B_k}}\bg_h\cdot\bn$ appears in the definition of
$L_1$. Actually, $\bg_h$
has been already defined in all the boundary nodes except for $B_k$ using the
values of $\bg$. Hence the notation $L_1(\bg)$ is consistent.

We define the value $\bg_h(B_k)$ such that
\begin{equation}
\label{eq: syst}\aligned
     \frac12(h_{k-1}+h_k)\bg_h(B_k)\cdot\bn &=
L_1(\bg)\\
\frac12(h_{k-1}+h_k)\bg_h(B_k)\cdot{\bf t} &=0,
     \endaligned
\end{equation}
where ${\bf t}$ denotes the unit tangential vector on $\Gamma$.
Taking into account that $\bg$ satisfies the compatibility condition, we have
\[
\aligned
L_1(\bg) &= \int_{\Gamma\setminus(\Gamma_V\cup\Gamma_{B_k})}(\bg-\bg_h)\cdot\bn
+\int_{\Gamma_V}(\bg-\bg_h)\cdot\bn +\int_{\Gamma_{B_k}}\bg\cdot\bn\\
&\qquad- \frac12\left[h_{k-1}\bg(B_{k-1})+h_k\bg(B_{k+1})\right]\cdot\bn.
\endaligned
\]
The first term can be bounded using standard results for interpolation errors on $\Gamma\setminus(\Gamma_V\cup\Gamma_{B_k})$.
To bound the other three terms, we use that
$\|\bg\|_{L^\infty(\Gamma)}\le \normg{1}$ and that the
length of the integration set is less than $h$. Then we obtain
\begin{equation}\label{eq: L1}
     \left|L_1(\bg)\right|\le C h \normg{1}.
\end{equation}
It is easy to check that the matrix of the system \eqref{eq: syst}
(for $\bg_h(B_k)$) is nonsingular and its inverse has norm of order $h^{-1}$.
So that we have
\begin{equation}
\label{eq: estim gh(Bk)}
     |\bg_h(B_k)|\le C\normg{1},
\end{equation}
where $|\bg_h(B_k)|$ stands for the Euclidean norm of the vector $\bg_h(B_k)$.
Therefore $\bg_h$ is defined on the entire $\Gamma$ and satisfies the
compatibility condition and the bound~\eqref{eq: norminfgh2d}.
\end{proof}

In the proof of the next proposition, we will use the embedding inequality for
$0\le s<\frac12$,
\begin{equation}\label{eq: embedding}
\|\phi\|_{L^q(\Gamma)}\le C_s\|\phi\|_{s,\Gamma}, \qquad \forall
\phi\in H^s(\Gamma)
\end{equation}
with  $q=\frac2{1-2s}$ and

\begin{equation}
\label{C1s}
C_s\sim \sqrt{\frac1{1-2s}}
\qquad \mbox{when } s\to \left(\frac12\right)^-.
\end{equation}
Inequality~\eqref{eq: embedding} is proved in~\cite[Theorem 1.1]{CT} in
$\RE$. The analogous result follows for an interval, and therefore for
$\Gamma$, by using an extension theorem.

\begin{proposition}
\label{pr: g-g_h}
For all $0\le s<\frac12$ we have
$$
\|\bg-\bg_h\|_{-s,\Gamma}\le \frac{C}{\sqrt{1-2s}} h^{\frac12+s} \normg{1}
$$
\end{proposition}

\begin{proof}
Let us set $p=\frac2{1+2s}$ and $q=\frac2{1-2s}$ its dual exponent.
Using  the H\"older inequality and the embedding
inequality~\eqref{eq: embedding}, we have
\begin{equation}
\label{eq: H_s}
\aligned
\|\bg-\bg_h\|_{-s,\Gamma} &=
\sup_{\phi: \|\phi\|_{s,\Gamma}=1}\int_\Gamma\left(\bg-\bg_h\right)\phi\\
&\le\sup_{\phi:
|\phi\|_{s,\Gamma}=1}\|\bg-\bg_h\|_{L^p(\Gamma)}\|\phi\|_{L^q(\Gamma)}\\
&\le C_s\|\bg-\bg_h\|_{L^p(\Gamma)}.
\endaligned
\end{equation}
Since $\bg_h$ coincides with the Lagrange interpolation of $\bg$ on
$\Gamma\setminus\left(\Gamma_V\cup\Gamma_{B_k}\right)$,
$|\Gamma_V\cup\Gamma_{B_k}|\le Ch$, and $1<p<2$, we have
\[
\aligned
\|\bg-\bg_h\|_{L^p(\Gamma)}^p&=
\|\bg-\bg_h\|_
{L^p\left(\Gamma\setminus\left(\Gamma_V\cup\Gamma_{B_k}\right)\right)}^p
+\|\bg-\bg_h\|_{L^p\left(\Gamma_V\cup\Gamma_{B_k}\right)}^p\\
&\le 
Ch^p\|\bg\|_{W^{1,p}(\Gamma\setminus\left(\Gamma_V\cup\Gamma_{B_k}\right))}^p 
+ Ch\|\bg\|_{L^\infty\left(\Gamma_V\cup\Gamma_{B_k}\right)}^p\\
&\le Ch\sum_{i=1}^{N_e}\|\bg\|_{H^1(\Gamma_i)}^p
\endaligned
\]
which together with \eqref{eq: H_s} yields,
\[
\|\bg-\bg_h\|_{-s,\Gamma}\le C\, C_s h^\frac1p \normg{1} \, .
\]
Using \eqref{C1s} and recalling that $p=\frac2{1+2s}$, we conclude the proof.
\end{proof}

In the next proposition we obtain a quasi-uniform in $h$ estimate
of the $H^\frac12$-norm of $\bg_h$.
\begin{proposition}
\label{pr: est_g_h}
If the family of meshes $\T_{\Gamma,h}$ is quasi-uniform we have
\[
\|\bg_h\|_{\frac12,\Gamma}\le C|\log h|\,\normg{1}.
\]
\end{proposition}
\begin{proof}
Let $\tilde\bg_h$ the Carstensen approximation. Then, for $0<s<1/2$,
inverse estimates imply
\[
\|\bg_h\|_{\frac12,\Gamma}\le \|\bg_h-\tilde\bg_h\|_{\frac12,\Gamma}+\|\tilde\bg_h\|_{\frac12,\Gamma}
\le C\left( h^{-{\frac12}}\|\bg_h-\tilde\bg_h\|_{0,\Gamma}+ h^{s-\frac12}\|\tilde\bg_h\|_{s,\Gamma}\right)
\]
and so, by \eqref{eq: Apel?} and the fact that $\bg\in\bH^s(\Gamma)$,
\[
\aligned
\|\bg_h\|_{\frac12,\Gamma}
&\le C\left( h^{-{\frac12}}\|\bg_h-\tilde\bg_h\|_{0,\Gamma}+
h^{s-\frac12}\|\bg\|_{s,\Gamma}\right)\\
&
\le C  h^{-{\frac12}}\left(\|\bg_h-\bg\|_{0,\Gamma}+\|\bg-\tilde\bg_h\|_{0,\Gamma}+
h^{s}\|\bg\|_{s,\Gamma}\right).
\endaligned
\]
Using Proposition \ref{pr: g-g_h}, \eqref{eq: Apel_s},
and~\eqref{eq: normsGamma}, we obtain
\[
\|\bg_h\|_{\frac12,\Gamma}
\le C \left(\normg{1} + h^{s-\frac12}\|\bg\|_{s,\Gamma}\right)
\le C  \frac{h^{s-\frac12}}{1-2s}\normg{1}.
\]
Choosing $s$ such that $1-2s=1/|\log h|$ we conclude the proof.
\end{proof}

\begin{remark}
\label{remark sobre cota gh en 12}
The quasi-uniformity assumption in the previous proposition is not essential.
We have given the proof under this hypothesis for the sake of simplicity.
However, for a general family of meshes, an elementary but rather technical
computation using the definition of the fractional norm leads to the estimate
$$
\|\bg_h\|_{\frac12,\Gamma}\le C|\log(h_{min})|\,\normg{1}.
$$
where $h_{min}$ denotes the minimum mesh-size of $\T_{\Gamma,h}$.
\end{remark}
\subsection{Three dimensional case}
We assume that the boundary $\Gamma$ is composed by $\Ne$ polygonal faces
$\Gamma_i$ and that $\bg|_{\Gamma_i}\in\bH^2(\Gamma_i)$. Therefore
$\bg\in\bL^\infty(\Gamma)$ and $\|\bg\|_{L^\infty(\Gamma)}\le C\normg{2}$.
Moreover, we can show as in the two dimensional case, that $\bg\in\bH^s(\Gamma)$
for $0\le s<\frac12$ and that
\begin{equation}
\label{eq: normsGamma3}
\|\bg\|_{s,\Gamma}\le \frac{C}{1-2s}\normg{2}.
\end{equation}
Assume that we have a triangular mesh $\mathcal T_{\Gamma,h}$ which is
quasi-uniform. A construction, similar to the one proposed here, can be 
made also in the case of quadrilateral quasi-uniform meshes.

As for the 2D case,
let $\{B_j\}$ be the set of nodes of $\mathcal T_{\Gamma,h}$ and define
\[
E= \bigcup\left\{\overline e: e \mbox{ is an edge of }\Omega\right\}.
\] 
For each node $B_j\in E$, let us choose $T_{B_j}$ an element of $\T_{\Gamma,h}$
such that $B_j\in\overline{T}_{B_j}$. 
Finally let $e_0$ be a polygonal contained in a face $\Gamma_k$ of $\Omega$, with
$|e_0|=O(1)$, made up of sides of triangles in $\mathcal T_{\Gamma,h}$ and such
that triangles with a vertex on $e_0$ do not have vertices on $E$, see
Fig.~\ref{fig: e0} for an example. It is clear that we can take it. We denote
by $\bn_{e_0}$ the normal vector to the face $\Gamma_k$ containing $e_0$.
\begin{figure}
\centering
\includegraphics[width=0.6\textwidth]{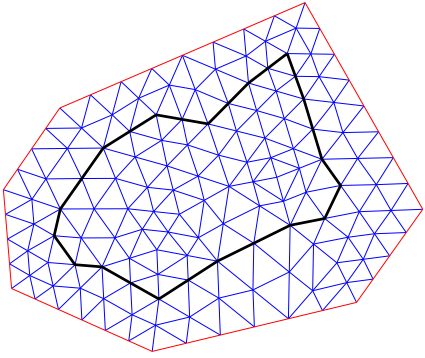}
\caption{A face of $\Omega$, with its mesh and the polygonal $e_0$ in black.} 
\label{fig: e0}
\end{figure}
\begin{lemma}
\label{le: gh3D}
Given $\bg\in \bL^2(\Gamma)$ such that $\bg|_{\Gamma_i}\in\bH^2(\Gamma_i)$
where $\Gamma_i$ for $i=1,\dots,\Ne$ are the faces of $\Gamma$,
there exists a piecewise linear function $\bg_h\in\Gh$ which is a modified
Lagrange interpolant of $\bg$ satisfying the compatibility
condition~\eqref{eq: compatibility} and
\begin{equation}
\label{eq: norminfgh3d}
\|\bg_h\|_{L^\infty(\Gamma)}\le C\normg{2}.
\end{equation}
\end{lemma}
\begin{proof}
We define the Lagrange interpolation $\bg_h\in\Gh$ of $\bg$ as the
continuous piecewise linear function on $\mathcal T_{\Gamma,h}$ such that for
each node $B_j$ in $\T_{\Gamma,h}$ we have
\[
\bg_h(B_j) = \left\{\begin{array}{cl}
\bg(B_j)&\mbox{if }B_j\not\in (E\cup e_0)\\
\bg|_{T_{B_j}}(B_j)&\mbox{if }B_j\in E\\
\balpha&\mbox{if }B_j\in e_0,
\end{array}\right.
\]   
where $\balpha$ is a vector to be chosen in order to verify the compatibility
condition \eqref{eq: compatibility}.

For a set $A\subset\Gamma$, we denote by $\omega_{\Gamma,A}$ the union of the
closures of the elements in $\mathcal T_{\Gamma,h}$ having a vertex on the closure of
$A$. Then we impose
\[
0  = \int_\Gamma\bg_h\cdot\bn
=\int_{\Gamma\setminus(\omega_{\Gamma,E}\cup\omega_{\Gamma,e_0})}\bg_h\cdot\bn
+ \int_{\omega_{\Gamma,E}}\bg_h\cdot\bn +
\int_{\omega_{\Gamma,e_0}}\bg_h\cdot\bn. 
\]
Let us compute the last term. Clearly, $\omega_{\Gamma,e_0}$ lays on the face
$\Gamma_k$ with normal $\bn_{e_0}$. Each triangle $T$ in $\omega_{\Gamma,e_0}$
has $r_T\ge 1$ vertices on $\overline{e}_0$, that we denote
$P_{T,1},\ldots,P_{T,r_t}$, while $P_{T,r_T+1},\ldots,P_{T,3}$ are the
remaining ones. Then
\[
\int_{\omega_{\Gamma,e_0}}\bg_h\cdot\bn
= \frac13\sum_{T\subset\omega_{\Gamma,e_0}}|T|r_T\balpha\cdot\bn_{e_0}
+ \frac13\sum_{T\subset\omega_{\Gamma,e_0}}|T|\sum_{i=r_T+1}^3\bg_h\cdot\bn_{e_0}
(P_{T,i})
\]
We require that the vector $\alpha$ is such that the following equality holds true
\[
\left(\frac13\sum_{T\subset\omega_{\Gamma,e_0}}|T|r_T\right)
\balpha\cdot\bn_{e_0}=L_1(\bg).
\]
where, taking into account that the continuous solution
satisfies~\eqref{eq: uno},
\[
\aligned
L_1(\bg):=
& \int_{\Gamma\setminus(\omega_{\Gamma,E}\cup\omega_{\Gamma,e_0})}
(\bg-\bg_h)\cdot\bn + \int_{\omega_{\Gamma,E}}(\bg-\bg_h)\cdot\bn\\ 
&\qquad\qquad
+ \int_{\omega_{\Gamma,e_0}} \bg\cdot\bn 
- \frac13\sum_{T\subset\omega_{\Gamma,e_0}}|T|\sum_{i=r_T+1}^3
  \bg_h\cdot\bn_{e_0}(P_{T,i}).
\endaligned
\]
Since $|\omega_{\Gamma,E}|$ and $|\omega_{\Gamma,e_0}|$ are bounded by $Ch$,
using interpolation error estimates, we see that 
\[
|L_1(\bg)|\le C\,h \normg{2}.
\]
In order to be able to find a unique $\balpha$, we add two conditions on the
tangential components obtaining the following system
\[
\aligned
\left(\frac13\sum_{T\subset\omega_{\Gamma,e_0}}|T|r_T\right)
\balpha\cdot\bn_{e_0}
& =L_1(\bg)\\ \balpha\cdot{\bf t}_1 &=0\\ \balpha\cdot{\bf t}_2 &=0
\endaligned
\] 
where ${\bf t}_1$ and $\bf t_2$ are unitary vectors such that together with
$\bn_{e_0}$ form an orthogonal basis of $\mathbb R^3$. This is a linear system
for $\balpha$ whose non singular matrix $M$ verifies
$\|M^{-1}\|\le C\frac1h$ 
since the mesh is quasi-uniform. Therefore, we can find $\balpha$ such that
\begin{equation}
\label{eq: stimaalpha3d}
|\balpha|\le C\normg{2}.
\end{equation}
This inequality, together with the definition of $\bg_h$,
gives~\eqref{eq: norminfgh3d}.
\end{proof} 
In the following proposition we estimate $\|\bg-\bg_h\|_{-s,\Gamma}$. Since
the best possible exponent $q$ in the embedding inequality
\eqref{eq: embedding} depends on the dimension, the
argument used in Proposition~\ref{pr: g-g_h} does not give an optimal
result in the case of a three dimensional domain.
We can give a different argument using a Hardy type inequality.
It will become clear that the same argument can be used for $d=2$, but it gives
a worse constant in terms of $s$ than that obtained in the
Proposition~\ref{pr: g-g_h}.
\begin{proposition}
\label{pr: g-gh3d}
There exists a positive constant, such that, for all $0\le s<\frac12$, the
following bound holds true
\begin{equation}
\label{g-gh 3d}
\|\bg-\bg_h\|_{-s,\Gamma}\le \frac C{1-2s}h^{\frac12+s}\normg{2}.
\end{equation}
\end{proposition}
\begin{proof}
For each $\Gamma_i$, face of $\Omega$, and $x\in\Gamma_i$, we denote by $d_i(x)$ the
distance of $x$ from $\partial\Gamma_i$. There exists a constant
$C$ such that, for $0\le s<\frac12$ and every $\phi\in H^s(\Gamma_i)$, we have
\begin{equation}
\label{hardy inequality}
\left\|\frac{\phi}{d_i^s}\right\|_{0,\Gamma_i} \le
\frac{C}{1-2s}\|\phi\|_{s,\Gamma_i}.
\end{equation}
This estimate with a precise constant is proved in \cite{BD} for the half-space,
by standard argument, one can show that the behavior of the constant
in terms of $s$ is the same for Lipschitz bounded domains.

For simplicity let us assume that the polygonal $e_0$ chosen in the
construction of $\bg_h$ is close to the boundary of 
$\Gamma_k$, i.e., if $x\in e_0$, then $d_k(x)\le C_1 h$ for some constant $C_1$.
Then, for any $\phi\in H^s(\Gamma)$,
$$
\int_\Gamma \left(\bg-\bg_h\right)\phi
=\sum_{i=1}^{N_e}\int_{\Gamma_i}\left(\bg-\bg_h\right)\phi\\
\le \sum_{i=1}^{N_e}\|(\bg-\bg_h)d_i^s\|_{0,\Gamma_i}
\left\|\frac{\phi}{d_i^s}\right\|_{0,\Gamma_i},
$$
and therefore, using \eqref{hardy inequality}, we obtain
$$
\|\bg-\bg_h\|_{-s,\Gamma} = \sup_{\phi:
\|\phi\|_s=1}\int_\Gamma\left(\bg-\bg_h\right)\phi \le
\frac{C}{1-2s}\sum_{i=1}^{N_e}\|(\bg-\bg_h)d_i^s\|_{0,\Gamma_i}.
$$
But,
\[
\aligned \|(\bg-\bg_h)d_i^s\|^2_{0,\Gamma_i}
&=\int_{\{x\in\Gamma_i: d_i(x)\le C_1h\}}(\bg-\bg_h)^2d_i^{2s}
+\int_{\{x\in\Gamma_i: d_i(x)> C_1h\}}(\bg-\bg_h)^2d_i^{2s}\\
&\le C \left(h^{2s+1} \|\bg\|^2_{L^\infty(\Gamma)} + h^2
\|\bg\|^2_{1,\Gamma_i}\right) \quad\text{for }i\ne k,\\
\|(\bg-\bg_h)d_k^s\|^2_{0,\Gamma_k}
&\le C \left(h^{2s+1} \|\bg\|^2_{L^\infty(\Gamma)}+h^{2s+1}\normg{2}^2
+ h^2 \|\bg\|^2_{1,\Gamma_k}\right) 
\endaligned
\]
where, for the first term, we have used that
$|\{x\in\Gamma_i: d_i(x)\le C_1h\}|\le Ch$, that 
$\|\bg_h\|_{L^\infty(\Gamma)}\le C\|\bg\|_{L^\infty(\Gamma)}$, and
inequality~\eqref{eq: stimaalpha3d}, while, for the second one, that
$\bg_h$ agrees with the Lagrange interpolation. Hence, we conclude
that, for all $0\le s<\frac12$, the bound~\eqref{g-gh 3d} holds true.
\end{proof}

The next proposition can be proved using the same argument as in
Proposition~\ref{pr: est_g_h}.
\begin{proposition}
\label{pr: est_g_h3d}
If the family of meshes $\T_{\Gamma,h}$ is quasi-uniform we have
\[
\|\bg_h\|_{\frac12,\Gamma}\le C|\log h|\,\normg{2}.
\]
\end{proposition}

We are ready to prove the main theorem of the section.

\begin{theorem}
Let $\Omega\subset\mathbb R^d$, $d=2$ or $3$, be a convex polygonal or polyhedral domain.  Suppose that
$\bg|_{\Gamma_i}\in H^{d-1}(\Gamma_i)$ for all $\Gamma_i$ and that the family of
meshes $\T_{\Gamma,h}$ is quasi-uniform. Let $\bg_h$ be given by the modified
Lagrange interpolation of $\bg$ introduced in Lemmas~\ref{le: gh2D}
and~\ref{le: gh3D}. Then, we have
\begin{enumerate}
\item For $\Omega\subset\RE^2$ a convex polygonal domain
\[
\|\bu-\bu_h\|_{0,\Omega}+\|p-p_h\|_{H^{-1}(\Omega)/\mathbb R} \le C h|\log h|^{\frac32} \normg{1},
\]
\item For $\Omega\subset\RE^3$ a convex polyhedral domain
\[
\|\bu-\bu_h\|_{0,\Omega}+\|p-p_h\|_{H^{-1}(\Omega)/\mathbb R} \le C h|\log h|^2 \normg{2}.
\]
\end{enumerate}
\end{theorem}
\begin{proof}
From Propositions \ref{pr: apriori}, \ref{pr: g-g_h} and~\ref{pr: g-gh3d} we
have, for $0\le s<\frac12$,
\[
\|\bu-\bu(h)\|_{0,\Omega}+\|p-p(h)\|_{H^{-1}(\Omega)/\mathbb R}\le
\frac{C}{(1-2s)^{\frac{d+1}2}} h^{\frac12+s} \normg{d-1}.
\]
Then, taking $s=1/2+1/\log h<1/2$ yields
\begin{equation}
\label{eq: main1}
\|\bu-\bu(h)\|_{0,\Omega}+\|p-p(h)\|_{H^{-1}(\Omega)/\mathbb R}\le C\,h\,
|\log h|^\frac{d+1}2 \normg{d-1}.
\end{equation}
On the other hand, from Propositions \ref{pr: err_h}, \ref{pr: est_g_h}
and~\ref{pr: est_g_h3d}, we have
$$
\|\bu(h)-\bu_h\|_{0,\Omega}+\|p(h)-p_h\|_{H^{-1}(\Omega)/\mathbb R}\le
Ch\left|\log h\right|\normg{d-1}
$$
which together with \eqref{eq: main1} gives the desired estimates.
\end{proof}
\begin{remark}
In view of Remark \ref{remark sobre cota gh en 12}, the quasi-uniformity assumption
in the previous theorem can be removed obtaining, for general regular family
of meshes, the analogous estimates with $|\log h|$ replaced by $|\log(h_{min})|$.
\end{remark}

\section{A posteriori error estimates}
\label{se: aposteriori}
\setcounter{equation}{0}
In this section we introduce the error indicator for the finite element
solution of our problem and show that it provides upper and lower bounds for
the discretization error of the regularized problem.

We denote by $\E_h$ the union of the interior edges/faces of the
elements of the mesh $\T_h$, and define
\[
\bJ:\E_h\to\RE^d, \qquad \bJ|_e=\bJ_e
\quad \mbox{with } \bJ_e=\bjump{\frac{\partial\bu_h}{\partial\bn}-p_h\bn}_e
\text{ for }e\in\E_h
\]
where the jump of the function $r$ across the edge $e=T^+\cap T^-$ is given by
\[
\bjump{\frac{\partial\bu_h}{\partial\bn}-p_h\bn}_e=\left(\frac{\partial\bu_h|_{T^+}}{\partial\bn^+}-p_h|_{T^+}\bn^+\right)+\left(\frac{\partial\bu_h|_{T^-}}{\partial\bn^-}-p_h|_{T^-}\bn^-\right)
\]
if $\bn^\pm$ denotes the exterior normal to the triangle $T^\pm$.

Then we introduce the local error indicator
\begin{equation}
\label{estim}
\eta_T^2=h_T^4\|-\Delta\bu_h+\nabla p_h\|_{0,T}^2+h_T^2\|\diver\bu_h\|_{0,T}^2
+\frac12\sum_{e\subset \overline T}h_T^3\|\bJ_e\|_{0,e}^2.
\end{equation}
Since we want to estimate the velocity in the $L^2(\Omega)$-norm and the pressure in
the $H^{-1}(\Omega)/\RE$-norm, the error indicator results to be the
usual error indicator for problems with smooth boundary data multiplied
by $h_T^2$ (see, e.g., \cite{V1989,F}).
\begin{proposition}[Robustness]
\label{pr: robust}
The estimator $\eta_T$ introduced in \eqref{estim} is robust, that is, there
exists a positive constant $C$ independent of $h$ such that
\begin{equation}
\label{eq: rosbust}
\|\bu(h)-\bu_h\|_{0,\Omega}+\|p(h)-p_h\|_{H^{-1}(\Omega)/\mathbb R}
\le C \left(\sum_{T\in\T_h}\eta_T^2\right)^\frac12.
\end{equation}
\end{proposition}
\begin{proof}
We start with the estimate for $\bu(h)-\bu_h$. In order to apply a duality
argument, we consider the solution $(\bphi,q)$ of \eqref{aux3}.
Then, taking into account the equations~\eqref{cuatro} and~\eqref{eq: weak}, and
the approximation estimates~\eqref{eq: interp} and~\eqref{eq: pd},
we obtain by integration by parts:
\begin{equation}
\label{seis}
\aligned
\|\bu(h)-\bu_h\|_{0,\Omega}^2 &= \left(\bu(h)-\bu_h,\bu(h)-\bu_h\right)
= \left(\bu(h)-\bu_h,-\Delta\bphi+\nabla q\right)\\
&=\left(\nabla(\bu(h)-\bu_h),\nabla(\bphi-\bphi^I)\right)\\
&\qquad-\left(\diver(\bu(h)-\bu_h),q-P_0q\right)
-\left(p(h)-p_h,\diver(\bphi-\bphi^I\right)\\
& =\sum_{T\in\T_h}\left(\left(\Delta\bu_h,\bphi-\bphi^I\right)_T
-\left(\frac{\partial\bu_h}{\partial\bn},\bphi-\bphi^I\right)_{\partial T}
\right)\\
&\qquad +(\diver\bu_h,q-P_0q)\\
&\qquad
-\sum_{T\in\T_h}\left(\left(\nabla p_h,\bphi-\bphi^I\right)_T
-\left(p_h,(\bphi-\bphi^I)\cdot\bn\right)_{\partial T}\right)\\
&=-\sum_{T\in\T_h}\left(-\Delta\bu_h+\nabla p_h,\bphi-\bphi^I\right)_T
+\sum_{T\in\T_h}\left(\diver\bu_h,q-P_0q\right)_T\\
&\qquad-\sum_{e\in\E_h}\left(
\bjump{\frac{\partial\bu_h}{\partial\bn}-p_h\bn},\bphi-\bphi^I\right)_e.
\endaligned
\end{equation}
Thanks to~\eqref{eq: interp}, we can write
\begin{equation}
\label{eq: rob1}
\aligned
\|\bu(h)-\bu_h\|_{0,\Omega}^2&\le C
\sum_{T\in\T_h}\left\|-\Delta\bu_h+\nabla p_h\right\|_{0,T}h_T^2|\bphi|_{2,T}\\
&\quad
+C\sum_{T\in\T_h}\|\diver\bu_h\|_{0,T}h_T|q|_{1,T}
+C\sum_{e\in\E_h}
\left\|\bJ\right\|_{0,e}h_T^\frac32\left|\bphi\right|_{2,\omega_e}
\\
&\le C\left[\sum_{T\in\T_h}\left(h_T^4\|-\Delta\bu_h+\nabla p_h\|_{0,T}^2
+h_T^2\|\diver\bu_h\|_{0,T}^2\right)\right.\\
&\quad\left.
+\sum_{e\in\E_h}h_T^3\|\bJ\|_{0,e}^2\right]^\frac12\|\bu(h)-\bu_h\|_{0,\Omega}\\
& \le C\left(\sum_{T\in\T_h} \eta_T^2\right)^{\frac12}\|\bu(h)-\bu_h\|_{0,\Omega}
\endaligned
\end{equation}
where $\omega_e$ is the union ot the elements sharing $e\in\E_h$.
This concludes the estimate of $\|\bu(h)-\bu_h\|_{0,\Omega}$.

Now we consider the error for the pressure. Since $p(h)$ and $p_h$ have zero mean
value, the definition of the $H^{-1}$-norm reads
\begin{equation}
\label{eq: defnorm-1}
\|p(h)-p_h\|_{H^{-1}(\Omega)/\mathbb R} =
\sup_{\substack{q\in H^1_0(\Omega)\\\int_\Omega q=0}}
\frac{\left(p(h)-p_h,q\right)}{\|q\|_{1,\Omega}}.
\end{equation}
For each $q\in H^1_0(\Omega)$ with $\int_\Omega q=0$, we take
$\bpsi\in H^2_0(\Omega)$ with $\diver\psi=q$ and $\|\bpsi\|_{H^2(\Omega)}\le
C\|q\|_{H^1(\Omega)}$ see~\eqref{aux2}, hence
\[
\aligned
\left(p(h)-p_h,q\right)&=\left(p(h)-p_h,\diver\bpsi\right)\\
&=\left(p(h)-p_h,\diver(\bpsi-\bpsi^I)\right)
-\left(\nabla(\bu(h)-\bu_h),\nabla(\bpsi-\bpsi^I)\right)\\
&\qquad-\left(\bu(h)-\bu_h,\Delta\bpsi \right).
\endaligned
\]
By the same computations performed in equation~\eqref{seis}, we obtain
\[
\aligned
\left(p(h)-p_h,q\right) &=\left(\Delta\bu(h)-\nabla p(h),\bpsi-\bpsi^I\right) \\
&\qquad +\sum_{T\in\mathcal T_h}\left(-\left(\Delta\bu_h,\bpsi-\bpsi^I\right)_T
+\left(\frac{\partial\bu_h}{\partial\bn},\bpsi-\bpsi^I\right)_T\right)\\
&\qquad +\sum_{T\in\T_h}\left(\left(\nabla p_h,\bpsi-\bpsi^I\right)_T-
\left(p_h,(\bpsi-\bpsi^I)\cdot\bn\right)_{\partial T}\right)\\
&\qquad -\left(\bu(h)-\bu_h,\Delta\bpsi\right)\\
&=\sum_{T\in\T_h}\left(-\Delta\bu_h+\nabla p_h,\bpsi-\bpsi^I\right)_T\\
&\qquad +\sum_{e\in\E_h}\left(\bjump{\frac{\partial\bu_h}{\partial\bn}-p_h\bn},
\bpsi-\bpsi^I\right)_e -\left(\bu(h)-\bu_h,\Delta\bpsi\right).
\endaligned
\]
Then
\[
\aligned
\left(p(h)-p_h,q\right) &\le
C\bigg[\sum_{T\in\T_h} h_T^4\|-\Delta\bu_h+\nabla p_h\|_{0,T}^2+
\sum_{e\in\E_h}h_T^3\|\bJ\|_{0,e}^2\bigg]^\frac12\|q\|_{1,\Omega}\\
&\qquad + \|\bu(h)-\bu_h\|_{0,\Omega}\|q\|_{1,\Omega}.
\endaligned
\]
The proof concludes by using the estimate \eqref{eq: rob1} and the norm
definition~\eqref{eq: defnorm-1}.
\end{proof}

In the next proposition we show that the error indicator bounds locally
the error by below.

\begin{proposition}[Efficiency]
    \label{pr: efi}
    For all element $T\in\T_h$, we have
    \begin{equation}
    \label{efi}
    \eta_T\le C\left(\|\bu(h)-\bu_h\|_{0,\omega_T}
    + \|p(h)-p_h\|_{-1,\omega_T}\right)
    \end{equation}
    where $\omega_T=\left\{T'\in\mathcal T_h: \overline{T'}\cap \overline{T}\ne\emptyset\right\}$.
\end{proposition}
\begin{proof}
    We estimate the three terms of the error indicator in~\eqref{estim},
    separately.
    Given an element $T\in\T_h$, let us consider the function
    \[
        b_T=\left(\prod_{i=1}^{d+1}\lambda_{i,T}\right)^2
    \]
    with $\lambda_{i,T}, i=1,\ldots,d+1$ being the
    barycenter coordinate functions in $T$. We set
    \[
    \bw_T=\left(-\Delta\bu_h+\nabla p_h\right)b_T.
    \]
    Thanks to the definition of $b_T$ we have that
    \[
    \bw_T=0 \quad\mbox{on }\partial T, \qquad
    \nabla \bw_T = 0\quad\mbox{on }\partial T,
    \]
    and, by inverse inequality,
    \begin{equation}
    \label{eq: stimawt}
    \aligned
    \|\diver\bw_T\|_{1,T}&
    \le Ch_T^{-2}\left\|-\Delta\bu_h+\nabla p_h\right\|_{0,T}\\
    \|\Delta\bw_T\|_{0,T}&\le
    Ch_T^{-2}\left\|-\Delta\bu_h+\nabla p_h\right\|_{0,T},
    \endaligned
    \end{equation}
    Then integration by parts gives
    \begin{equation}
    \label{eq: effuno}
    \aligned
    \|-\Delta\bu_h+\nabla p_h\|_{0,T}^2&=
    \left(-\Delta\bu_h+\nabla p_h,-\Delta\bu_h+\nabla p_h\right))_T\\
    &\le C\left|\left(-\Delta\bu_h+\nabla p_h,\bw_T\right)_T\right|\\
    &=C\left|\left(-\Delta(\bu_h-\bu(h))+\nabla(p_h-p(h)),\bw_T\right)\right|\\
    &=C\left|\left(\bu_h-\bu(h),\Delta\bw_T\right)_T
    +\left(p_h-p(h),\diver\bw_T\right)\right|.
    \endaligned
    \end{equation}
    Due to the definition of $b_T$ we have that $\diver\bw_T\in\bH^1_0(T)$, hence we can use the duality between $H^{-1}(T)$ and $H^1_0(T)$ to obtain
    \[
    \|-\Delta\bu_h+\nabla p_h\|_{0,T}^2
    \le C\left(\|\bu_h-\bu(h)\|_{0,T}\|\Delta\bw_T\|_{0,T}
    +\|\left(p_h-p(h)\right)\|_{-1,T}\|\diver\bw_T\|_{1,T}\right)
    \]
    which, together with \eqref{eq: stimawt}, implies
    \begin{equation}
    \label{ef1}
    h_T^2\left\|-\Delta\bu_h+\nabla p_h\right\|_{0,T}\le
    C\left(\|\bu_h-\bu(h)\|_{0,T} + \|p_h-p(h)\|_{-1,T}\right).
    \end{equation}

    In order to bound the second term in~\eqref{estim}, let us introduce
    $w_T=(\diver\bu_h)b_T$, which satisfies
    \[
    \|\nabla w_T\|_{0,T}\le Ch_T^{-1}\|\diver\bu_h\|_{0,T}.
    \]
    Hence we obtain
    \[
    \aligned
    \int_T\left(\diver\bu_h\right)^2 &
\le C\left|\int_T\left(\diver\bu_h\right)w_T\right|
    = C\left|\int_T\diver\left(\bu_h-\bu(h)\right)w_T\right|\\
    &=C\left|\int_T\left(\bu_h-\bu(h)\right)\nabla w_T\right|
    \le Ch_T^{-1}\|\bu_h-\bu(h)\|_{0,T}\|\diver\bu_h\|_{0,T}
    \endaligned
    \]
    which implies
    \begin{equation}
    \label{ef2}
    h_T\|\diver\bu_h\|_{0,T}\le C\|\bu_h-\bu(h)\|_{0,T}.
    \end{equation}
It remains to bound the last term of the indicator involving the jumps along
element interfaces in $\T_h$. Let $e\in\E_h$ be an internal edge/face
and let us suppose that there are two elements $T_1$ and $T_2$ such that
$e=T_1\cap T_2$. Let $\tt v_i$ for $i=1,\dots,d$, be the vertices of $e$. We
denote by $\lambda_{\tt v_i,T_j}$, $i=1,\dots,d$, $j=1,2$, the barycentric
coordinate functions for the vertex $\tt v_i$ on the triangle $T_j$ and by
$\omega_e$ the union of $T_1$ and $T_2$. Then we define the bubble function
    \[
    b_e=\left(\prod_{i=1}^d\lambda_{{\tt v}_i,T_1}
    \prod_{i=1}^d\lambda_{{\tt v}_i,T_2}\right)^2.
    \]
    Setting $\bw_e=\bJ_eb_e$ and taking into account that the mesh is regular,
    it is not difficult to check that the following inequalities hold true:
    \begin{equation}
    \label{eq: stimaw}
    \aligned
    &\|\Delta\bw_e\|_{0,\omega_e}\le Ch_e^{-\frac32}\|\bJ_e\|_{0,e}\\
    &\|\diver\bw_e\|_{1,\omega_e}\le Ch_e^{-\frac32}\|\bJ_e\|_{0,e}\\
    &\|\bw_e\|_{0,\omega_e}\le Ch_e^\frac12\|\bJ_e\|_{0,e}.
    \endaligned
    \end{equation}
    There exists a positive constant $C$ such that
    \[
    \aligned
    \frac1C\|\bJ_e\|_{0,e}^2&\le (\bJ_e^2,b_e)_e
    =\left(\bjump{\frac{\partial\bu_h}{\partial\bn}-p_h\bn},\bw_e\right)_e\\
    &=\left(\nabla\bu_h,\nabla \bw_e\right)_{\omega_e} +
    \sum_{T\subset\omega_e}(\Delta\bu_h,\bw_e)_T
    -\sum_{T\subset\omega_e}(\nabla p_h,\bw_e)_T
    - (p_h,\diver\bw_e)_{\omega_e}\\
    &= (\nabla\bu_h,\nabla \bw_e)_{\omega_e} - (p_h,\diver\bw_e)_{\omega_e}
    + \sum_{T\subset\omega_e}\left(\Delta\bu_h-\nabla p_h,\bw_e\right)_T\\
    &=\left(\nabla\bu_h-\nabla\bu(h),\nabla\bw_e\right)_{\omega_e}
    -\left(p_h-p(h),\diver\bw_e\right)_{\omega_e}\\
    &\qquad+\sum_{T\subset\omega_e}\left(\Delta\bu_h-\nabla p_h,\bw_e\right)_T\\
    &=-\left(\bu_h-\bu(h),\Delta\bw_e\right)_{\omega_e}
    +\left(\bu_h-\bu(h),\frac{\partial\bw_e}{\partial\bn}\right)_{\partial\omega_e}\\
    &\qquad -\left(p_h-p(h),\diver\bw_e\right)_{\omega_e}
    +\sum_{T\subset\omega_e}\left(\Delta\bu_h-\nabla p_h,\bw_e\right)_T.
    \endaligned
    \]
    Using again the fact that $\diver\bw_e\in H^1_0(\omega_e)$, we obtain, by
multiplying times $h_e^3$
    \[
    \aligned
    h_e^3\|\bJ_e\|_{0,e}^2&\le
    C\Big( \|\bu_h-\bu(h)\|_{0,\omega_e}h_e^3\|\Delta\bw_e\|_{0,\omega_e}
    + \|p_h-p(h)\|_{-1,\omega_e}h_e^3\|\diver\bw_e\|_{1,\omega_e} \\
    &\qquad+ \sum_{T\subset\omega_e}h_e^2\|\Delta\bu_h-\nabla p_h\|_{0,T}
    h_e\|\bw_e\|_{0,T}\Big).
    \endaligned
    \]
    Using~\eqref{eq: stimaw} and~\eqref{ef1}, we have
    \[
        \aligned
        h_e^3\|\bJ_e\|_{0,e}^2&\le
        C\Big( \|\bu_h-\bu(h)\|_{0,\omega_e}
        + \|p_h-p(h)\|_{-1,\omega_e} \\
        &\qquad+ \sum_{T\subset\omega_e}h_e^2\|\Delta\bu_h-\nabla p_h\|_{0,T}
        \Big)h_e^\frac32\|\bJ\|_{0,e}\\
        &\le \left(\|\bu_h-\bu(h)\|_{0,\omega_e}
        + \|p_h-p(h)\|_{-1,\omega_e}\right) h_e^\frac32\|\bJ\|_{0,e},
        \endaligned
    \]
    that is,
    \begin{equation}
    \label{ef3}
    h_e^\frac32\|\bJ_e\|_{0,e}\le
    C\left(\|\bu_h-\bu(h)\|_{0,\omega_e} + \|p_h-p(h)\|_{-1,\omega_e}\right).
    \end{equation}
    Taking into account the definition \eqref{estim} of the estimator $\eta_T$, together with the estimates \eqref{ef1}, \eqref{ef2} and \eqref{ef3} we obtain the desired result.
\end{proof}

\section{Numerical Example}
\label{se: examples}
\setcounter{equation}{0}
On $\Omega=[0,1]^2$, we consider the lid-driven cavity flow problem
\begin{eqnarray*}
-\Delta{\bu}+\nabla p &=& 0\qquad \mbox{in }\Omega\\
\mbox{div\,}\bu&=& 0\qquad\mbox{on }\Omega\\
\bu&=&{\bf g}\qquad \mbox{in }\partial\Omega
\end{eqnarray*}
with
\[
{\bf g}(x_1,x_2) = \left\{\begin{array}{cl}(1,0)&\mbox{if }0<x_1<1\mbox{ and }x_2=1\\(0,0)&\mbox{if }x_1=0 \mbox{ or }x_1=1\mbox{ or }x_2=0.\end{array}\right.
\]
We consider the methods:
\begin{itemize}
\item Mini-element: ${\bf V}_h=(\mathcal P_1^b(\mathcal T_h))^2\cap \mathcal C^0(\bar \Omega)^2$ and $Q_h=\mathcal P_1(\mathcal T_h)\cap \mathcal C^0(\bar \Omega)\cap L^2_0(\Omega)$
\item Hood and Taylor: ${\bf V}_h=\mathcal P_2(\mathcal T_h)^2\cap \mathcal C^0(\bar \Omega)^2$ and $Q_h=\mathcal P_1(\mathcal T_h)\cap \mathcal C^0(\bar \Omega)\cap L^2_0(\Omega)$
\end{itemize}
where, if for an element $T$, $b_T\in\mathcal P_3$ is the cubic bubble function vanishing on $\partial T$, we set
\[
\mathcal P_1^b(T) = \mathcal P_1(T)\oplus \mbox{span}\left\{b_T(\cdot)\right\}.
\]
We consider the variational formulation \eqref{eq: pd} with
$\bu_h=E\bg_h+\bu_{0h}$ where $\bu_{0h}\in{\bf V}_h$ and $\bg_h$ is the
Lagrange interpolation of $\bg$ on the restriction of $\mathcal T_h$ to
$\partial\Omega$. We remark that the compatibility condition \eqref{eq:
compatibility} is automatically verified.

Below for the distinct methods and different refinement strategies we estimate
the convergence errors for $\bu$ in $L^2(\Omega)$-norm.
Since we do not know the exact solution, the
$L^2(\Omega)$-error is computed as the difference between the solutions
obtained at two consecutive refinements.

\begin{table}[ht]
\centering
\begin{tabular}{|c|c|c|c|c|}
\hline
nv & $L^2$ error in $\bu$ & $\eta$ & order in $\bu$ &  order in $\eta$\\\hline
$289  $&  $0.051393  $&  $1.8518  $&  &   \\\hline
$1089  $&  $0.025876  $&  $0.93123  $&  $0.51724  $&  $0.51818$\\\hline
$4225  $&  $0.012952  $&  $0.46698  $&  $0.51049  $&  $0.5091$\\\hline
$16641  $&  $0.0064768  $&  $0.23386  $&  $0.50553  $&  $0.50449$\\\hline
$66049  $&  $0.0032384  $&  $0.11703  $&  $0.50281  $&  $0.5022$\\\hline
\end{tabular}
\caption{Mini-element on uniformly refined structured meshes.}\label{tabla1}
\end{table}

\begin{table}[ht]
\centering
\begin{tabular}{|c|c|c|c|c|}
\hline
nv & $L^2$ error in $\bu$ & $\eta$ & order in $\bu$ & order in $\eta$\\\hline
$289  $&  $0.04065  $&   $3.603  $&   &  \\\hline
$1089  $&  $0.020324  $&    $1.8041  $&  $0.52253  $&    $0.52142$\\\hline
$4225  $&  $0.010162  $&  $0.90276  $&  $0.51127  $&    $0.51068$\\\hline
$16641  $&  $0.0050809  $&    $0.45158  $&  $0.50563  $&   $0.50532$\\\hline
\end{tabular}
\caption{Hood-Taylor on uniformly refined structured meshes.}\label{tabla2}
\end{table}

Tables \ref{tabla1} and \ref{tabla2} show results obtained by uniform
refinements starting with a coarse mesh for Mini-element and Hood-Taylor
methods respectively. We observe that, in both cases, order $\frac12$ with
respect to the number of elements (order $1$ in $h$) is obtained for the error
decay in $L^2(\Omega)$ of $\bu$. Accordingly, the error estimator $\eta$
defined by
\begin{equation}\label{eq: total_estim}
\eta^2=\sum_{T\in\mathcal T_h}\eta_T^2,
\end{equation}
with $\eta_T$ given by \eqref{estim}, decreases with the same order.

\begin{table}[ht]
\centering
\begin{tabular}{|c|c|c|c|c|}
\hline
nv & $L^2$ error in $\bu$ & $\eta$ & order in $\bu$ & order in $\eta$\\\hline
$84  $&  $0.04161  $&  $2.4824  $&  &   \\\hline
$99  $&  $0.036736  $&  $1.5566  $&  $0.75833  $&  $2.8409$\\\hline
$107  $&  $0.017576  $&  $1.0568  $&  $9.4869  $&   $4.9828$\\\hline
$148  $&  $0.015596  $&  $0.67029  $&  $0.36853  $&  $1.4036$\\\hline
$201  $&  $0.010888  $&  $0.4429  $&  $1.174  $&  $1.3537$\\\hline
$316  $&  $0.0064777  $&  $0.27551  $&  $1.1477  $&  $1.0493$\\\hline
$514  $&  $0.0041708  $&  $0.17191  $&  $0.90498  $&  $0.96945$\\\hline
$778  $&  $0.0025481  $&  $0.11492  $&  $1.1888  $&  $0.9716$\\\hline
$1197  $&  $0.0017212 $&  $0.07334  $&  $0.91061  $&  $1.0425$\\\hline
$1901  $&  $0.0011844 $&  $0.04755  $&  $0.80805  $&  $0.93682$\\\hline
$2859  $&  $0.00074065  $&  $0.030819  $&  $1.1504  $&  $1.0626$\\\hline
$4416  $&  $0.00050373  $&  $0.01985  $&  $0.88666  $&  $1.0118$\\\hline
$6834  $&  $0.00033144  $&  $0.01284  $&  $0.95861  $&  $0.99762$\\\hline
$10248  $&  $0.00021719  $&  $0.0084216  $&  $1.0431  $&  $1.041$\\\hline
$15443  $&  $0.00014117  $&  $0.0055306  $&  $1.0507  $&  $1.0254$\\\hline
\end{tabular}
\caption{Adaptive scheme for the Mini-element method using the local
estimators $\eta_T$. Parameter: $\theta=0.5$.}\label{tabla3}
\end{table}

\begin{table}[ht]
\centering
\begin{tabular}{|c|c|c|c|c|}
\hline
nv & $L^2$ error in $\bu$ & $\eta$ & order in $\bu$ & order in $\eta$\\\hline
$84  $&  $0.035974  $&  $4.2984  $&  &   \\\hline
$91  $&  $0.02702  $&  $2.3946  $&  $3.5759  $&  $7.3089$\\\hline
$102  $&  $0.015866  $&  $1.3891  $&  $4.6655  $&  $4.772$\\\hline
$118  $&  $0.0082372  $&  $0.82406  $&  $4.4988  $&  $3.5837$\\\hline
$160  $&  $0.0045485  $&  $0.46452  $&  $1.9503  $&  $1.8826$\\\hline
$237  $&  $0.0024787  $&  $0.25907  $&  $1.5452  $&  $1.4862$\\\hline
$385  $&  $0.0014542  $&  $0.1406  $&  $1.0991  $&  $1.2597$\\\hline
$636  $&  $0.00078342  $&  $0.075834  $&  $1.2322  $&  $1.2299$\\\hline
$992  $&  $0.00042274  $&  $0.041933  $&  $1.3878  $&  $1.3328$\\\hline
$1615  $&  $0.00026069  $&  $0.022199  $&  $0.99186  $& $1.3051$\\\hline
$2583  $&  $0.00012981  $&  $0.011858  $&  $1.4848  $&  $1.3352$\\\hline
$4154  $&  $7.3921e-005  $&  $0.0062463  $&  $1.1851  $& $1.3491$\\\hline
$6665  $&  $3.9114e-005  $&  $0.0032513  $&  $1.3463  $&  $1.381$\\\hline
$10447  $&  $1.9238e-005 $&  $0.0017148  $&  $1.5788  $&  $1.4235$\\\hline
$16629  $&  $1.0136e-005  $&  $0.00089668  $&  $1.3785  $&  $1.3948$\\\hline
$26283  $&  $5.559e-006  $&  $0.00046444  $&  $1.3122  $&  $1.4371$\\\hline
$40802  $&  $2.7269e-006  $&  $0.00024283  $&  $1.6195  $&  $1.4744$\\\hline
$64222  $&  $1.3931e-006  $&  $0.00012827  $&  $1.4806  $&  $1.4069$\\\hline
\end{tabular}
\caption{Adaptive scheme for the Hood-Taylor method using the local estimators
$\eta_T$. Parameter: $\theta=0.75$.}\label{tabla4}
\end{table}

In Tables \ref{tabla3} and \ref{tabla4} we show the results obtained by an adaptive procedure using the a posteriori error estimator \eqref{eq: total_estim}. The refinement process is standard: given $0<\theta<1$, a fixed parameter, suppose that $\mathcal T_k$ is the mesh in the $k$-step. If we enumerate the triangular elements such that $\mathcal T_k = \{T_i: i=1,\ldots,N_{el}\}$ with $\eta_{T_i}\ge\eta_{T_{i+1}}$, let $N_{ref,k}$ be the minimum integer such that
\[
    \sum_{i=1}^{N_{ref,k}}\eta_{T_i}^2\ge\theta\,\eta^2.
\]
Then, the mesh for the $k+1$-step is constructed in such a way that the
elements $T_{i}$, $i=1,\ldots,N_{ref,k}$ are refined. We report the
$L^2(\Omega)$-error in $\bu$ which, as before, is
computed in each step as the $L^2(\Omega)$-norm of the difference between the
discrete solution obtained in the current step and in the previous one of the
iterative process.

We observe that for both Mini-element and Hood-Taylor methods, the adaptive
process recovers the expected optimal order of convergence in $\bu$. In Figure
\ref{Figura_TH} we show the initial mesh and some of the meshes obtained in the
iterative process for Hood-Taylor method.

\begin{figure}
\centering
\includegraphics[width=.48\textwidth]{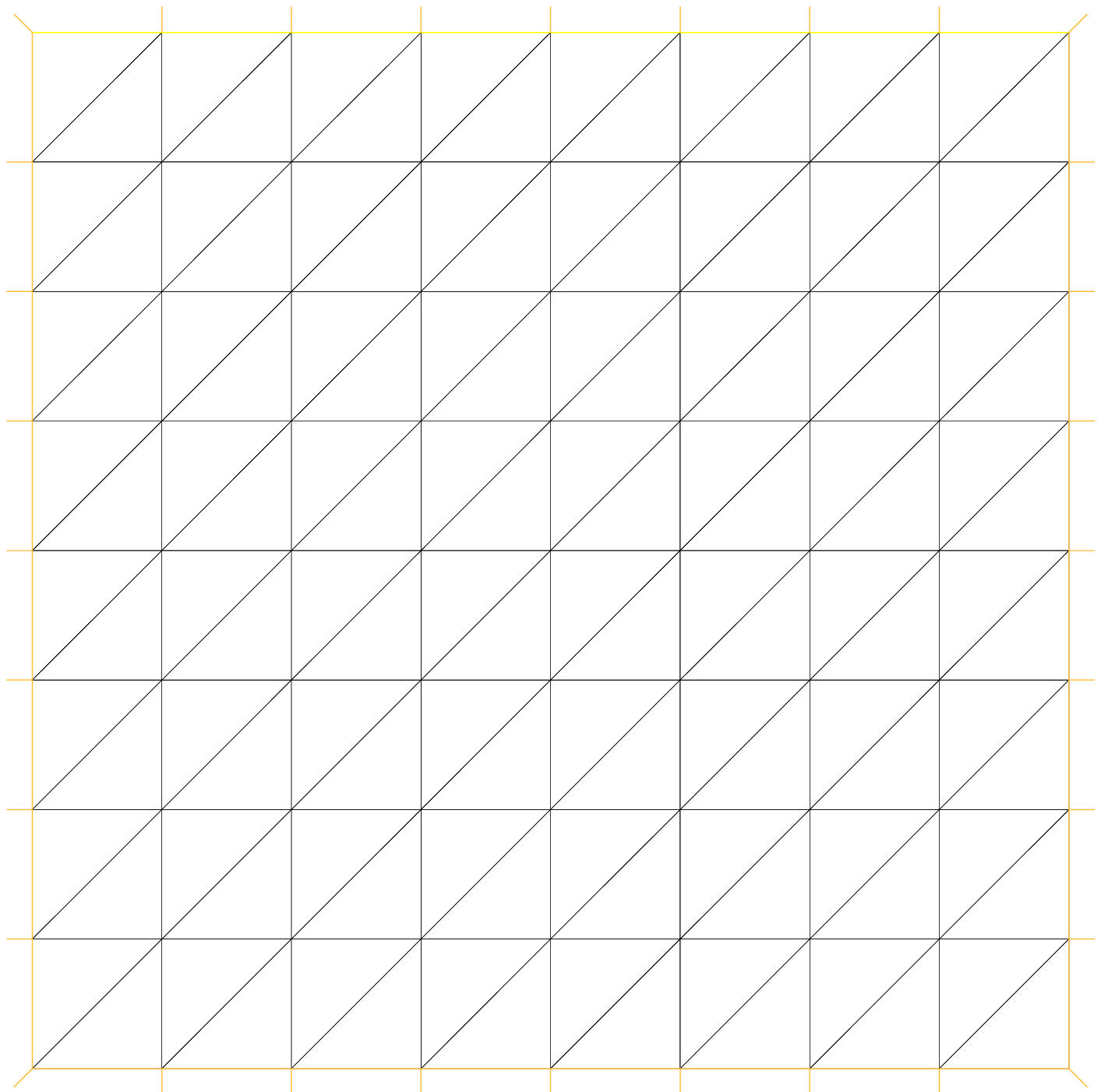}
\includegraphics[width=.48\textwidth]{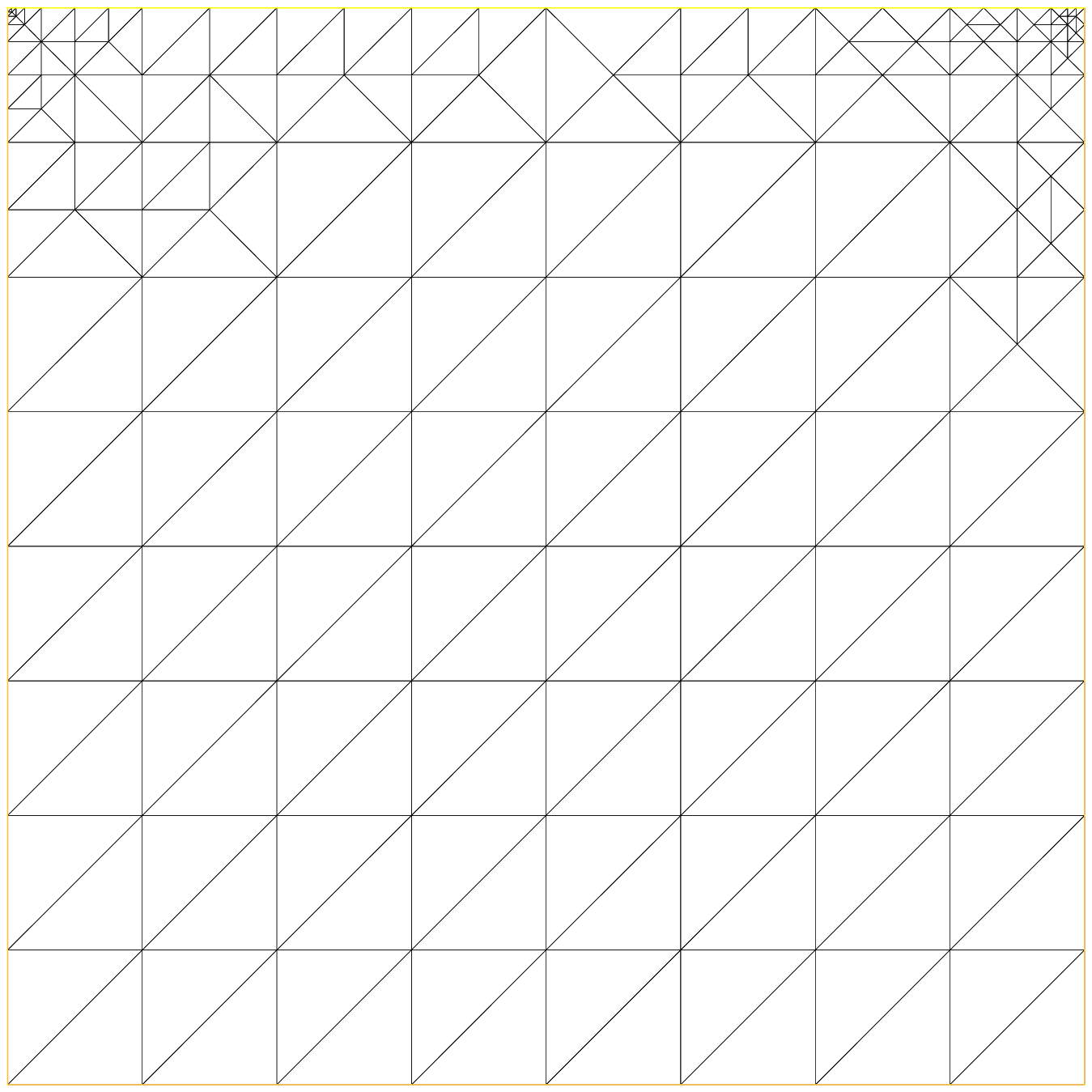}
\includegraphics[width=.48\textwidth]{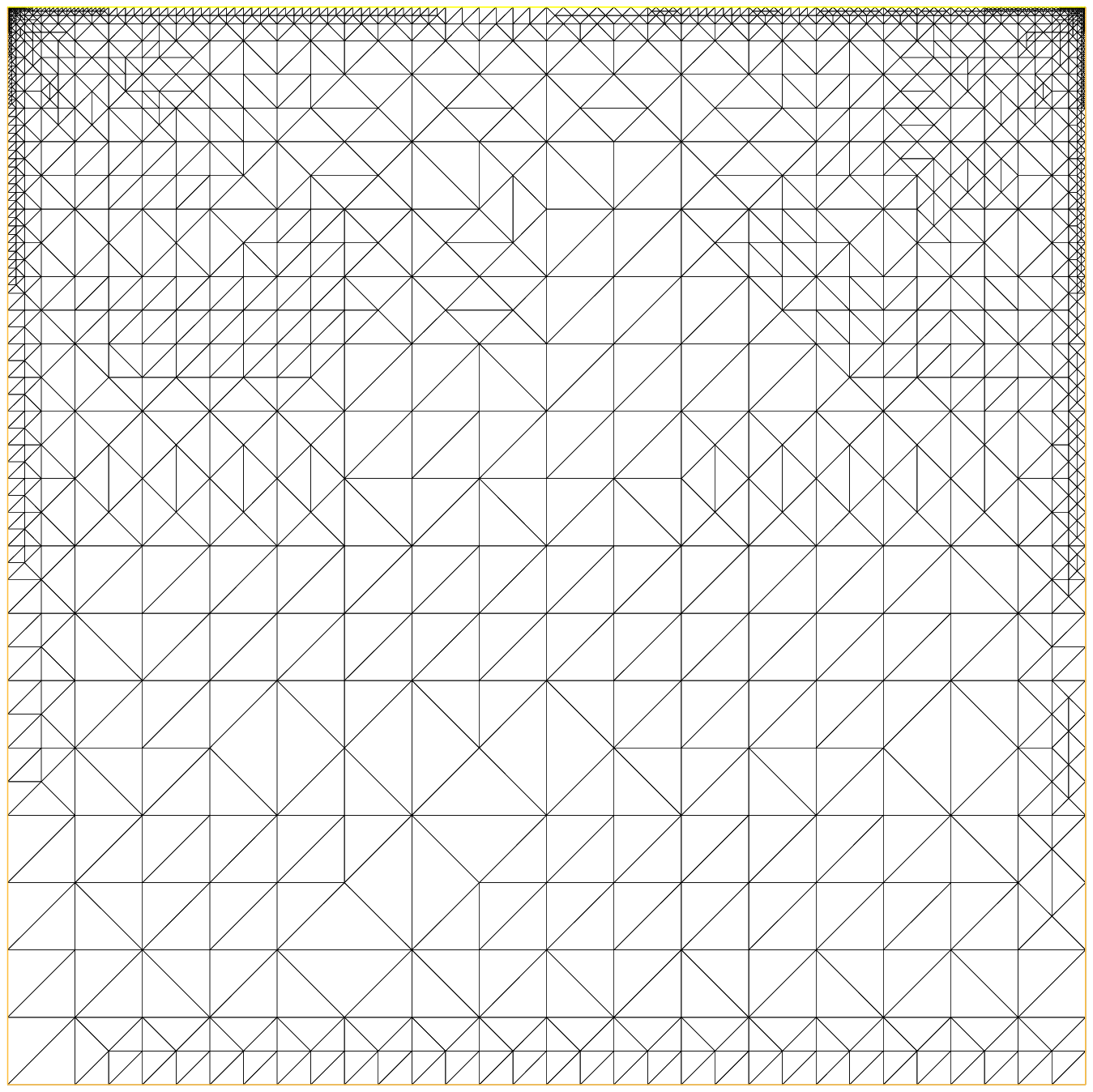}
\includegraphics[width=.48\textwidth]{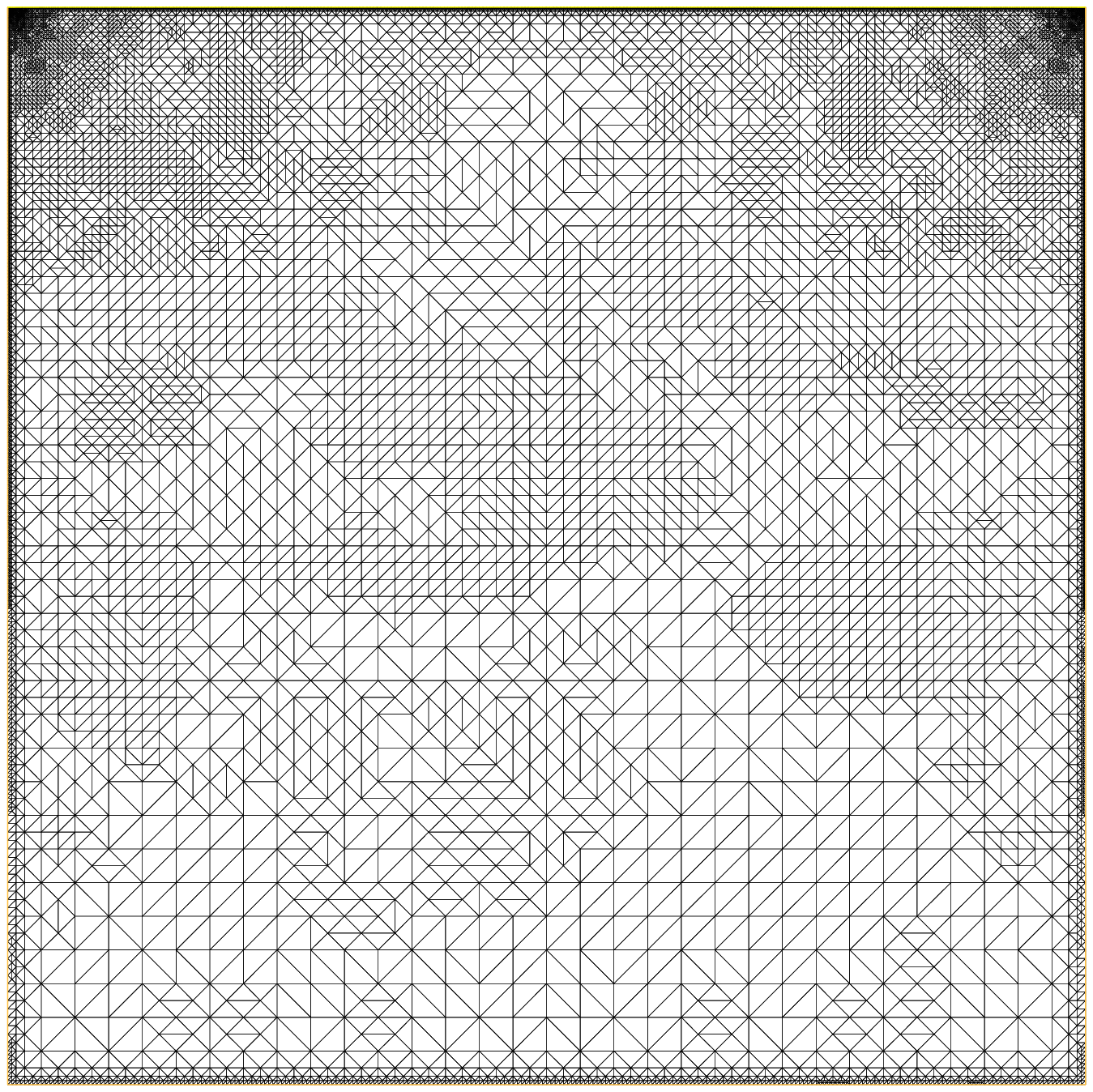}
\caption{Sequence of meshes for the Hood-Taylor adaptive process using the
local estimators $\eta_T$ with parameter $\theta=0.75$. Initial mesh and meshes
of iterations 5, 10 and 15.}\label{Figura_TH}
\end{figure}

\section*{Acknowledgments}
We thank Pablo De Napoli who suggested us the argument used in Proposition
3.3, and Pedro Morin for helpful discussions.

\bibliographystyle{plain}
\bibliography{ref}

\end{document}